\documentclass[showkeys]{amsart}
\usepackage[style=alphabetic,backend=bibtex]{biblatex} 
\usepackage{scalerel}

\calclayout

\usepackage{amscd,verbatim}
\usepackage{amssymb}
\usepackage[all]{xy}
\usepackage{preample}
\usepackage{eucal}
\usepackage{microtype}
\usepackage{xcolor}
\usepackage{hyperref}
\hypersetup{
 colorlinks,
 linkcolor={teal},
 citecolor={teal},
 urlcolor={teal}
}
\DeclareFieldFormat
  [article,book,inbook,incollection,inproceedings,patent,thesis,unpublished]
  {title}{\emph{#1\isdot}}
 
 \renewcommand{\bcube}{{\overline{\square}}}

\let\Pro=\relax
\let\PSh=\relax

\let\CI=\relax
\let\ch=\relax
\DeclareMathOperator{\CDiv}{CDiv}

\DeclareMathOperator{\Nm}{Nm}

\DeclareMathOperator{\Pro}{Pro}
\DeclareMathOperator{\PSh}{PSh}
\DeclareMathOperator{\Sh}{Sh}
\DeclareMathOperator{\CI}{CI}

\DeclareMathOperator{\Adm}{Adm}

\DeclareMathOperator{\ch}{ch}

\def\int{\mathrm{int}}

\def\intHom{\mathcal{H}\mathrm{om}}

\def\smallcube{{\scaleobj{0.8}{\overline{\square}}}}
\def\Lcube{h_0^{\smallcube}}
\def\Rcube{h^{0,\smallcube}}
\def\aff{\mathrm{aff}}

\def\category#1{\operatorname{\mathrm{#1}}}
\def\Sm{\category{Sm}}

\def\Comp{\category{Comp}}

\def\Cor{\category{Cor}}

\def\RSC{\category{RSC}}
\def\Ab{\category{Ab}}

\def\MCor{\category{MCor}}
\def\ulMCor{\category{\underline{M}Cor}}

\makeatletter
\def\quotprojlim{%
  \mathop{``\mathpalette\varlim@{\leftarrowfill@\textstyle}"}\nmlimits@
}
\def\quotinjlim{%
  \mathop{``\mathpalette\varlim@{\rightarrowfill@\textstyle}"}\nmlimits@
}
\makeatother

\begin{document}
\title{A motivic construction of the de Rham-Witt complex}

\author[J. Koizumi]{Junnosuke Koizumi}
\address{Graduate School of Mathematical Sciences, University of Tokyo, 3-8-1 Komaba, Meguro-ku, Tokyo 153-8914, Japan}
\email{jkoizumi@ms.u-tokyo.ac.jp}

\author[H. Miyazaki]{Hiroyasu Miyazaki}
\address{NTT Institute for Fundamental Mathematics, NTT Communication Science Laboratories, Nippon Telegraph and Telephone Corporation, 3-9-11 Midori-cho, Musashino-shi, Tokyo 180-8585, Japan}
\email{hiroyasu.miyazaki@ntt.com}

\date{\today}
\thanks{The first author is supported by JSPS KAKENHI Grant (22J20698).
The second author is supported by JSPS KAKENHI Grant (19K23413, 21K13783). 
}



\begin{abstract}
The theory of reciprocity sheaves due to Kahn-Saito-Yamazaki is a powerful framework to study invariants of smooth varieties via invariants of pairs $(X,D)$ of a variety $X$ and a divisor $D$.
We develop a generalization of this theory where $D$ can be a $\mathbb{Q}$-divisor.
As an application, we provide a motivic construction of the de Rham-Witt complex, which is analogous to the motivic construction of the Milnor $K$-theory due to Suslin-Voevodsky.
\end{abstract}


\maketitle
\setcounter{tocdepth}{1}
\tableofcontents

\enlargethispage*{20pt}
\thispagestyle{empty}

\section*{Introduction}

In Voevodsky's theory of mixed motives, the notion of $\mathbb{A}^1$-invariant sheaf played a fundamental role (see \cite{voetri}, \cite{VoeA1}, \cite{MVW} etc.).
An $\mathbb{A}^1$-invariant sheaf is a \emph{sheaf with transfers} satisfying \emph{$\mathbb{A}^1$-invariance}. 
A sheaf with transfers is a Nisnevich sheaf of abelian groups on the category of \emph{finite correspondences}, denoted $\Cor_k$, where $k$ is the fixed base field.
The objects of $\Cor_k$ are smooth schemes over $k$, and a morphism $X \to Y$ in $\Cor_k$ is given by an algebraic cycle on $X \times Y$ whose components are finite surjective over a connected component of $X$.
We say that a sheaf with transfers $F$ is $\mathbb{A}^1$-invariant if $\pr_1^*\colon F(X) \to F(X \times \mathbb{A}^1)$ is an isomorphism for any $X \in \Cor_k$.

On the other hand, in a series of papers \cite{KSY1}, \cite{KSY2}, Kahn-Saito-Yamazaki developed the theory of \emph{reciprocity sheaves}.
This is a vast generalization of the theory of $\mathbb{A}^1$-invariant sheaves over a field, and it captures ramification-theoretic information of invariants of schemes (see e.g. \cite{RS21}).
The class of reciprocity sheaves includes many interesting examples that are not $\mathbb{A}^1$-invariant, such as the sheaf of differential forms, the Hodge-Witt sheaf, and all commutative algebraic groups.

Let us recall what reciprocity sheaves are.
The key idea is to replace smooth schemes by {\it proper modulus pairs}.
A proper modulus pair over $k$ is a pair $(X,D_X)$ of a proper $k$-scheme $X$ and an effective Cartier divisor $D_X$ on $X$ such that $X \setminus |D_X|$ is smooth over $k$.
For example, the pair $\bcube :=(\mathbb{P}^1,[\infty])$ is a proper modulus pair, which we call the \emph{cube}. 
We can define a category of proper modulus pairs $\MCor_k$ similar to $\Cor_k$ by taking into account the information of Cartier divisors.
An additive presheaf $F\colon (\MCor_k)^\op\to \Ab$ is said to be \emph{cube invariant} if for any modulus pair $\sX = (X,D_X)$, the map 
\[
\pr_1^*\colon F(\sX) \to F(\sX \otimes \bcube), \quad \sX \otimes \bcube := (X \times \mathbb{P}^1, \pr_1^*D_X +\pr_2^*[\infty])
\]
is an isomorphism.
A presheaf with transfers (resp. sheaf with transfers) $F$ is said to be a reciprocity presheaf (resp. reciprocity sheaf) if it belongs to the essential image of
\[
\left(
	\begin{tabular}{c}
	\text{cube invariant}\\
	\text{presheaves}
	\end{tabular}
\right)
\hookrightarrow
\PSh(\MCor_k)
\xrightarrow{\omega_!}
\PSh(\Cor_k),
\]
where $\omega_!$ is the left Kan extension of $(X,D_X) \mapsto X \setminus |D_X|$
\footnote{
Originally, the notion of reciprocity sheaf given in \cite{KSY1} looks quite different from the above one. 
However, in \cite{KSY2}, it is shown that the above definition coincides with the original one, provided that the base field $k$ is perfect.
An advantage of the above description is that we can think of a reciprocity sheaf as the ``shadow'' of a cube invariant presheaf. 
Since the definition of cube invariant presheaves is very similar to that of $\mathbb{A}^1$-invariant presheaves, one can prove many important properties of cube invariant presheaves by following Voevodsky's classical methods, at least partly.
}.

In the first half of this paper, we will generalize the theory of modulus pairs and sheaves on them to allow the effective divisor $D_X$ to have \emph{rational coefficients} (see applications below for the reason why we need this generalization).
We will define the category of \emph{proper $\mathbb{Q}$-modulus pairs} $\MCor^\mathbb{Q}_k$ over $k$ (Definition \ref{def_QCor}).
We can define the notion of cube invariant presheaf and the functor $\omega_!\colon \PSh(\MCor^\mathbb{Q}_k)\to \PSh(\Cor_k)$ as before.
In fact, the resulting notion of $\mathbb{Q}$-reciprocity presheaf coincides with the usual one.
This allows us to use $\mathbb{Q}$-modulus pairs in the theory of reciprocity sheaves.

\

In the latter half of this paper, we will give some applications of our theory. 
The main results include a motivic construction of the de Rham-Witt complex, which we will sketch below.

It is well-known that the multiplicative group $\mathbb{G}_m$, which is an important example of an $\mathbb{A}^1$-invariant sheaf, has a motivic presentation. 
That is, there exists a canonical isomorphism $\mathbb{G}_m \simeq h_0^{\mathbb{A}^1}(\mathbb{Z}_\tr(\mathbb{A}^1\setminus\{0\})/\mathbb{Z})$ in $\PSh(\Cor_k^\aff)$, where $\Cor_k^\aff$ is the full subcategory of $\Cor_k$ consisting of affine schemes, $\mathbb{Z}_\tr\colon \Cor_k\to \PSh(\Cor_k)$ is the Yoneda embedding, and $h_0^{\mathbb{A}^1}$ is the $0$-th Suslin homology:
$$
	h_0^{\mathbb{A}^1}F(X)=\Coker(F(X\times \mathbb{A}^1)\xrightarrow{i_0^*-i_1^*} F(X)).
$$
Moreover, the group structure on $\mathbb{G}_m$ is simply induced by 
the multiplication morphism $(\mathbb{A}^1\setminus\{0\}) \times (\mathbb{A}^1\setminus\{0\}) \to \mathbb{A}^1\setminus\{0\}$.
Suslin-Voevodsky proved more generally that the sheaf of unramified Milnor $K$-theory admits a motivic construction \cite[Theorem 3.4]{SV00}.
These results are fundamental for various computations in the theory of mixed motives.

It is a natural idea to extend this to other invariants of schemes.
We start with a motivic presentation of the ring of big Witt vectors $\mathbb{W}_n(A)=1+tA[t]/(t^{n+1})$ of a ring $A$.
For $n\geq 0$, we define $\mathbb{W}_n^+\in \PSh(\MCor^\mathbb{Q}_k)$ by
$$
\mathbb{W}_n^+=\varinjlim_{\varepsilon>0}\mathbb{Z}_\tr(\mathbb{P}^1, (n+\varepsilon)[\infty])/\mathbb{Z},
$$
where $\mathbb{Z}_\tr\colon \MCor^\mathbb{Q}_k\to \PSh(\MCor^\mathbb{Q}_k)$ is the Yoneda embedding.
Then there are operations $F_s, V_s$ on $\mathbb{W}_n^+$ for each $s>0$ which are induced by the morphism $t\mapsto t^s$ on $\mathbb{A}^1$ and its transpose.
These operations are called the Frobenius and the Verschiebung.
Moreover, the multiplication on $\mathbb{A}^1$ induces a multiplication on $\mathbb{W}_n^+$.
The $0$-th Suslin homology $h_0^{\mathbb{A}^1}$ generalizes to $\PSh(\MCor^\mathbb{Q}_k)$:
$$
	\Lcube F(\mathcal{X})=\Coker(F(\mathcal{X}\otimes\bcube)\xrightarrow{i_0^*-i_1^*} F(\mathcal{X})).
$$
We write $h_0=\omega_!\Lcube$.
First we prove the following result.

\begin{introtheorem}[Theorem \ref{thm:ab-comparison}, Proposition \ref{prop:ring-comparison}]\label{thm:intro_1}
There is an isomorphism $h_0\mathbb{W}_n^+\xrightarrow{\sim}\mathbb{W}_n$ in $\PSh(\Cor_k^\aff)$ which preserves the multiplication, Frobenius, and the Verschiebung.
\end{introtheorem}

In the proof of this theorem, we use the fact that the group $h_0\mathbb{Z}_\tr(\mathbb{P}^1, r[\infty])\;(r\in \mathbb{Q}_{>0})$ depends only on $\lceil r \rceil$.
This is a consequence of the \textit{motivic Hasse-Arf theorem} which we prove in Theorem \ref{geometric Hasse-Arf for curves}.
It is an analogue of the classical Hasse-Arf theorem which states that the upper ramification group $G^{(r)}$ of an abelian extension of a local field depends only on $\lceil r \rceil$.

Theorem \ref{thm:intro_1} has a non-trivial application to reciprocity presheaves.
Imitating the construction in \cite{Miyazaki-19}, we define
$
NF(X)=\Ker(i_0^*\colon F(X\times\mathbb{A}^1)\to F(X))
$
for a reciprocity presheaf $F$.
Then we have $NF=0$ if and only if $F$ is $\mathbb{A}^1$-invariant.
Using the above theorem, we can prove that $\mathbb{W}=\varprojlim_{n\geq 0}\mathbb{W}_n$ acts on $NF$.
As a consequence, we obtain a short proof of the following result of Binda-Cao-Kai-Sugiyama \cite[Theorem 1.3]{BCKS}.

\begin{introtheorem}[Corollary \ref{BCKSgen}]
	Let $F$ be a reciprocity presheaf over $k$ and assume that $F$ is separated for the Zariski topology.
	\begin{enumerate}
		\item	If $\ch(k)=0$ and $F\otimes\mathbb{Q}=0$, then $F$ is $\mathbb{A}^1$-invariant.
		\item	If $\ch(k)=p>0$ and $F$ is $p$-torsion-free, then $F$ is $\mathbb{A}^1$-invariant.
	\end{enumerate}
\end{introtheorem}

Next we give a motivic presentation of the ring of $p$-typical Witt vectors $W_n(A)$ of a $\mathbb{Z}_{(p)}$-algebra $A$.
For $n\geq 0$, we define $\widehat{\mathbb{W}}_n^+\in \PSh(\MCor^\mathbb{Q}_k)$ by
$$
\widehat{\mathbb{W}}_n^+=\varinjlim_{\varepsilon>0}\mathbb{Z}_\tr(\mathbb{P}^1, \varepsilon[0]+(n+\varepsilon)[\infty]).
$$
As in the case of $\mathbb{W}_n^+$, we have operations $F_s,V_s$ for $s>0$ and a multiplication on $\widehat{\mathbb{W}}_n^+$.
We define $W_n^+$ to be the quotient of $\widehat{\mathbb{W}}_{p^{n-1}}^+\otimes \mathbb{Z}_{(p)}$ by the images of the idempotents $\ell^{-1}V_\ell F_\ell$ for all prime numbers $\ell\neq p$.
The operations $F_p,V_p$ descend to operations $F,V$ on $W_n^+$ called the Frobenius and the Verschiebung, and the multiplication also descends to $W_n^+$.
We prove the following

\begin{introtheorem}[Corollary \ref{cor:p_typical_comparison}]
There is an isomorphism $h_0W_n^+\xrightarrow{\sim}W_n$ in $\PSh(\Cor_k^\aff)$ which is compatible with the Frobenius, the Verschiebung, and the multiplication.
\end{introtheorem}

In order to treat the de Rham-Witt complex, we have to assume that $k$ is perfect and $p\geq 3$.
We define $\mathbb{G}_m^+\in \PSh(\MCor^\mathbb{Q}_k)$ by
$$
\mathbb{G}_m^+=\varinjlim_{\varepsilon>0}\mathbb{Z}_\tr(\mathbb{P}^1, \varepsilon[0]+\varepsilon[\infty])/\mathbb{Z}.
$$
Using the tensor structure on $\PSh(\MCor^\mathbb{Q}_k)$ (see \S 2), we obtain an object $W_n^+ \otimes \mathbb{G}_m^{+\otimes q}$ of $\PSh(\MCor^\mathbb{Q}_k)$.
Our main result is the following:

\begin{introtheorem}[Theorem \ref{thm:rep-WOmega}]
Let $k$ be a perfect field of characteristic $p\geq 3$.
Then there is an isomorphism
$$a_{\Nis} h_0 (W_n^+ \otimes \mathbb{G}_m^{+\otimes q})\xrightarrow{\sim}W_n\Omega^q$$
in $\Sh_\Nis(\Cor_k)$, where $a_\Nis$ denotes the Nisnevich sheafification.
\end{introtheorem}

The proof goes as follows: first we prove that the left hand side admits a Witt complex structure (Theorem \ref{thm:theta}). 
Since the de Rham-Witt complex is initial in the category of Witt complexes, we obtain a unique morphism from the right hand side to the left hand side.
To prove that this morphism is an isomorphism, we construct an inverse by using the transfer structure on $W_n\Omega^q$. 

We note that a similar motivic presentation for $\Omega^q$ is obtained by R\"ulling-Sugiyama-Yamazaki \cite{RSY} and the first author \cite{Koi2}.
Also, a presentation of the big de Rham-Witt complex $\mathbb{W}_n\Omega^q$ using additive higher Chow groups is obtained by R\"ulling \cite{Rulling-thesis} for fields, and by Krishna-Park \cite{KP11} for regular semilocal algebras over a  field.
We expect that these results will be connected by some comparison isomorphism between Suslin homology and additive higher Chow groups.

\

The structure of the present paper is as follows.
In \S \ref{sec:preliminiary}, 
we prepare preliminary results on $\mathbb{Q}$-Cartier divisors.
In \S \ref{sec:modpair},
we define $\mathbb{Q}$-modulus pairs and prove some basic properties.
In \S \ref{sec:ci-rec},
we define and study the notion of cube invariant presheaf and the notion of reciprocity presheaf following \cite{KSY2}. 
In \S \ref{sec:HA},
we formulate and prove the motivic Hasse-Arf theorem, by comparing the $0$-th Suslin homology group and the Chow group of relative $0$-cycles of a modulus curve. This comparison can be seen as a generalization of the result in \cite{RY16}.
In \S \ref{sec:Witt},
we apply our machinery to give a motivic construction of the ring of Witt vectors and basic operations on it.
As an application, we give a short proof of the result of Binda-Cao-Kai-Sugiyama \cite{BCKS} on torsion and divisibility of a reciprocity sheaf.
In \S \ref{sec:dRWitt},
we give a motivic construction of the de Rham-Witt complex.

\subsection*{Acknowledgements}

We would like to thank Shuji Saito for his interest on this work. We also appreciate his valuable comments on earlier versions of this paper, which encouraged the authors to improve the main results.

\subsection*{Notations and conventions}
\begin{itemize}
	\item	For a scheme $X$ and $x\in X$, we write $k(x)$ for the residue field of $X$ at $x$.
			We write $X^{(d)}$ for the set of points on $X$ of codimension $d$.
			An element of $X^{(0)}$ is called a \emph{generic point}.	
	\item	We say that a morphism of schemes $f\colon X\to Y$ is \emph{pseudo-dominant} if it takes generic points to generic points, i.e., $f(X^{(0)})\subset Y^{(0)}$.
	\item	For an integral scheme $X$, we write $k(X)$ for its function field.
			If $X$ is a noetherian normal integral scheme and $f\in k(X)^\times$, we write $\div(f)$ for the Weil divisor on $X$ defined by $f$.
	\item	We write $\Sm_k$ for the category of smooth separated $k$-schemes of finite type.
			We define $\Sm_k^\aff$ to be the full subcategory of $\Sm_k$ spanned by affine schemes.
	\item	For an additive category $\mathcal{C}$, we write $\PSh(\mathcal{C})$ for the category of additive functors $\mathcal{C}^\op\to \Ab$.
	\item	For a category $\mathcal{C}$, we write $\Pro(\mathcal{C})$ for the category of pro-objects in $\mathcal{C}$.
			A {\it pro-functor} $F\colon \mathcal{C}\dashrightarrow\mathcal{D}$ is a functor $F\colon \mathcal{C}\to \Pro(\mathcal{D})$.
			A pro-functor $F\colon \mathcal{C}\dashrightarrow\mathcal{D}$ is said to be \emph{pro-left adjoint} to $G\colon \mathcal{D}\to \mathcal{C}$ if there is a natural isomorphism
			$\Hom_{\Pro(\mathcal{D})}(F(c),d)\simeq \Hom_{\mathcal{C}}(c,G(d))$.
\end{itemize}

\newpage
\section{Preliminaries}\label{sec:preliminiary}

\subsection{Finite correspondences}

First we recall Suslin-Voevodsky's category of finite correspondences $\Cor_k$.
It is an additive category having the same objects as $\Sm_k$, and its group of morphisms $\Cor_k(X,Y)$ is the group of algebraic cycles on $X\times Y$ whose components are finite pseudo-dominant over $X$.
The fiber product over $k$ gives a symmetric monoidal structure on $\Cor_k$.
For any morphism $f\colon X\to Y$ in $\Sm_k$, the graph of $f$ gives a morphism $f\colon X\to Y$ in $\Cor_k$.
If $f$ is finite pseudo-dominant, then the transpose of the graph of $f$ gives a morphism ${}^tf\colon Y\to X$ in $\Cor_k$.
For $S\in \Sm_k$ and smooth $S$-schemes $X,Y$, we write $\Cor_S(X,Y)\subset \Cor_k(X,Y)$ for the subgroup of cycles supported on $X\times_SY$.

A \emph{presheaf with transfers} is an additive presheaf $F\colon (\Cor_k)^\op\to \Ab$.
If $f\colon X\to Y$ is a finite pseudo-dominant morphism in $\Sm_k$ and $F\in \PSh(\Cor_k)$, then we write $f_*$ or $\Tr_{X/Y}$ for the map $({}^tf)^*\colon F(X)\to F(Y)$.
We say that $F$ is a \emph{Nisnevich sheaf} if for any $X\in \Sm_k$, the presheaf $F_X\colon U\mapsto F(U)$ on $X_\Nis$ is a Nisnevich sheaf.
We write $\Sh_\Nis(\Cor_k)\subset\PSh(\Cor_k)$ for the full subcategory spanned by Nisnevich sheaves.
The inclusion functor $\Sh_\Nis(\Cor_k)\to \PSh(\Cor_k)$ admits an exact left adjoint $a_\Nis\colon \PSh(\Cor_k) \to \Sh_\Nis(\Cor_k)$, so $\Sh_\Nis(\Cor_k)$ is Grothendieck abelian.

\subsection{$\mathbb{Q}$-Cartier divisors}

We introduce the notion of $\mathbb{Q}$-Cartier divisor over a general noetherian scheme, which will play a fundamental role in this paper.

\begin{definition}
For a noetherian scheme $X$, we write $\CDiv(X)$ for the group of Cartier divisors on $X$.
We define the group of \emph{$\mathbb{Q}$-Cartier divisors} on $X$ by
$\CDiv_{\mathbb{Q}} (X) := \CDiv (X) \otimes_{\mathbb{Z}} \mathbb{Q}$.
\end{definition}

If $X$ is a noetherian normal scheme, then the group $\CDiv (X)$ is embedded into the free abelian group of Weil divisors on $X$. In particular, $\CDiv (X)$ is a free abelian group in this case. 
Let $f\colon Y \to X$ be a morphism of noetherian schemes and $D$ be a $\mathbb{Q}$-Cartier divisor on $X$. 
We say that the pullback $f^*D$ of $D$ by $f$ \emph{exists} if there is an ordinary Cartier divisor $E$ on $X$ and $r \in \mathbb{Q}$ with $D = rE$ such that the pullback $f^* E$ exists. 
In this situation, we define $f^* D := rf^*E\in \CDiv_{\mathbb{Q}} (Y)$.
This does not depend on the choice of $E$ and $r$.

Let $X$ be a noetherian scheme and $D$ be a $\mathbb{Q}$-Cartier divisor on $X$.
We say that $D$ is \emph{$\mathbb{Q}$-effective} if there is an ordinary effective Cartier divisor $E$ on $X$ and $r\in \mathbb{Q}_{>0}$ such that $D = rE$ in $\CDiv_{\mathbb{Q}}(X)$.
For $\mathbb{Q}$-Cartier divisors $D,D'$, we write $D\leq D'$ if $D'-D$ is $\mathbb{Q}$-effective.

If an ordinary Cartier divisor $D$ is effective, then it is $\mathbb{Q}$-effective.
The converse is not true in general, but it is true for normal schemes:

\begin{lemma}\label{effective_vs_q-effective}
	Let $X$ be a noetherian normal scheme and $D$ be an ordinary Cartier divisor on $X$.
	Then $D$ is effective if and only if $D$ is $\mathbb{Q}$-effective.
\end{lemma}

\begin{proof}
	We may assume that $X=\Spec A$ with $A$ a normal domain and $D=\div(a)$ with $a\in \Frac(A)^\times$.
	If $D$ is $\mathbb{Q}$-effective, then there is some non-zero-divisor $b\in A$ and $n\geq 1$ such that $\div(a)=(1/n)\div(b)$ in $\CDiv_\mathbb{Q}(X)$.
	Since $X$ is normal, the map $\CDiv(X)\to \CDiv_\mathbb{Q}(X)$ is injective and hence $\div(a^n)=\div(b)$ in $\CDiv(X)$.
	This implies that $a^n-ub=0$ for some $u\in A^\times$, so $a\in A$ since $A$ is normal.
	Therefore $D$ is effective.
\end{proof}

\begin{corollary}\label{ineq_abs}
	Let $X$ be a noetherian scheme and $D,D'$ be ordinary Cartier divisors on $X$. 
	Then we have $D|_{X^N}\leq D'|_{X^N}$ as $\mathbb{Q}$-Cartier divisors if and only if $D|_{X^N}\leq D'|_{X^N}$ as ordinary Cartier divisors.\end{corollary}

\begin{definition}
	Let $X$ be a noetherian scheme and $D$ be a $\mathbb{Q}$-effective $\mathbb{Q}$-Cartier divisor on $X$.
	Suppose that $D=rE$ in $\CDiv_\mathbb{Q}(X)$ where $E$ is an ordinary effective Cartier divisor on $X$ and $r\in \mathbb{Q}_{>0}$.
	Then the support $|D|$ of $D$ is defined to be $|E|$.
	This is well-defined since $|E|=|nE|$ for any ordinary effective Cartier divisor $E$ and $n\geq 1$.
\end{definition}

\section{$\mathbb{Q}$-modulus pairs}\label{sec:modpair}
\subsection{$\mathbb{Q}$-Modulus pairs}

Recall from \cite{KMSY1} that a \emph{modulus pair} (over $k$) is a pair $\mathcal{X}=(X,D_X)$ where $X$ is a separated $k$-scheme of finite type and $D_X$ is an effective Cartier divisor on $X$ such that $X^\circ:=X\setminus|D_X|$ is smooth over $k$.

\begin{definition}\label{def_QCor}
	A {\it $\mathbb{Q}$-modulus pair} $\mathcal{X}$ is a pair $(X,D_X)$ where $X$ is a separated $k$-scheme of finite type and $D_X$ is a $\mathbb{Q}$-effective $\mathbb{Q}$-Cartier divisor on $X$, such that $X^\circ:=X\setminus|D_X|$ is smooth over $k$.
\end{definition}

In order to avoid confusion, we use the word ``$\mathbb{Z}$-modulus pairs'' to indicate modulus pairs in the sense of \cite{KMSY1}.
In what follows, we fix $\Lambda\in \{\mathbb{Z},\mathbb{Q}\}$ and develop the theory of $\Lambda$-modulus pairs.

An \emph{ambient morphism} of $\Lambda$-modulus pairs $f\colon \mathcal{X}\to \mathcal{Y}$ is a morphism of $k$-schemes $f\colon X\to Y$ such that $f(X^\circ)\subset Y^\circ$ and $D_X|_{X^N}\geq f^*D_Y|_{X^N}$ hold.
It is called \emph{minimal} if $D_X=f^*D_Y$ holds.

Let $\mathcal{X},\mathcal{Y}$ be $\Lambda$-modulus pairs over $k$ and let $V\subset X^\circ\times Y^\circ$ be an integral closed subscheme.
We say that $V$ is \emph{left proper} if the closure $\overline{V}$ of $V$ in $X\times Y$ is proper over $X$.
We say that $V$ is \emph{admissible} if $(\pr_1^*D_X)|_{\overline{V}^N}\geq (\pr_2^*D_Y)|_{\overline{V}^N}$ holds.
We write $\ulMCor^\Lambda_k(\mathcal{X},\mathcal{Y})$ for the subgroup of $\Cor_k(X^\circ,Y^\circ)$ consisting of cycles whose components are left proper and admissible.
For any ambient morphism $f\colon \mathcal{X}\to \mathcal{Y}$ over $k$, the graph of $f^\circ$ gives a morphism $f\colon \mathcal{X}\to \mathcal{Y}$ in $\ulMCor^\Lambda_k$.
If $f$ is a proper minimal ambient morphism such that $f^\circ$ is finite pseudo-dominant, then the transpose of the graph of $f^\circ$ gives a morphism ${}^tf\colon \mathcal{Y}\to \mathcal{X}$ in $\ulMCor^\Lambda_k$.
Note that for $\mathbb{Z}$-modulus pairs $\mathcal{X},\mathcal{Y}$, we have
$\ulMCor^\mathbb{Z}_k(\mathcal{X},\mathcal{Y}) = \ulMCor^\mathbb{Q}_k(\mathcal{X},\mathcal{Y})$
by Corollary \ref{ineq_abs}.

\begin{lemma}\label{lem:composition_admissible}
	Let $\mathcal{X},\mathcal{Y},\mathcal{Z}$ be $\Lambda$-modulus pairs over $k$ and $\alpha\in \ulMCor^\Lambda_k(\mathcal{X},\mathcal{Y})$, $\beta\in \ulMCor^\Lambda_k(\mathcal{Y},\mathcal{Z})$.
	Then we have $\beta\circ \alpha\in \ulMCor^\Lambda_k(\mathcal{X},\mathcal{Z})$, where $\circ$ denotes the composition in $\Cor_k$.
\end{lemma}

\begin{proof}
	For $\Lambda = \mathbb{Z}$, this is proved in \cite[Proposition 1.2.4 and Proposition 1.2.7]{KMSY1}.
	For $\Lambda = \mathbb{Q}$, we may replace $D_X,D_Y,D_Z$ by $nD_X,nD_Y,nD_Z$ for $n\in \mathbb{Z}_{>0}$, so we may assume that $\mathcal{X},\mathcal{Y},\mathcal{Z}$ are $\mathbb{Z}$-modulus pairs.
	Then we have $\ulMCor^\mathbb{Z}_k(\mathcal{X},\mathcal{Y}) = \ulMCor^\mathbb{Q}_k(\mathcal{X},\mathcal{Y})$ etc., so the claim follows from the case $\Lambda = \mathbb{Z}$.
\end{proof}

We define $\ulMCor^\Lambda_k$ to be the category of $\Lambda$-modulus pairs over $k$, where the morphisms are given by $\ulMCor^\Lambda_k(\mathcal{X},\mathcal{Y})$.
The canonical functor $\ulMCor^\mathbb{Z}_k\to \ulMCor^\mathbb{Q}_k$ is fully faithful by Corollary \ref{ineq_abs}.
We set
$$
	\mathcal{X}\otimes \mathcal{Y}=(X\times Y, \pr_1^*D_X+\pr_2^*D_Y).
$$
The next lemma shows that this gives a symmetric monoidal structure on $\ulMCor^\Lambda_k$:

\begin{lemma}
	Let $\mathcal{X}_1,\mathcal{Y}_1,\mathcal{X}_2,\mathcal{Y}_2\in \ulMCor^\Lambda_k$ and $\alpha\in \ulMCor^\Lambda_k(\mathcal{X}_1,\mathcal{Y}_1)$, $\beta\in \ulMCor^\Lambda_k(\mathcal{X}_2,\mathcal{Y}_2)$.
	Then we have $\alpha\times \beta \in \ulMCor^\Lambda_k(\mathcal{X}_1\otimes\mathcal{X}_2,\mathcal{Y}_1\otimes\mathcal{Y}_2)$.
\end{lemma}

\begin{proof}
	For $\Lambda = \mathbb{Z}$, this is proved in \cite[Lemma 2.1.3]{KMSY3}.
	For $\Lambda = \mathbb{Q}$, we may reduce to the case $\Lambda = \mathbb{Z}$ as in the proof of Lemma \ref{lem:composition_admissible}.
\end{proof}

We define $\MCor^\Lambda_k\subset\ulMCor^\Lambda_k$ to be the full subcategory consisting of \emph{proper} $\Lambda$-modulus pairs, i.e., $\mathcal{X}=(X,D_X)$ with $X$ proper over $k$.
There are natural functors
\begin{align*}
&\tau\colon \MCor^\Lambda_k\to \ulMCor^\Lambda_k;\quad \mathcal{X}\mapsto \mathcal{X},\\
&\underline{\omega}\colon \ulMCor^\Lambda_k\to \Cor_k;\quad \mathcal{X}\to X^\circ.
\end{align*}

A \emph{$\Lambda$-modulus presheaf} is an additive presheaf $F\colon (\ulMCor^\Lambda_k)^\op\to \Ab$.
The category $\PSh(\ulMCor^\Lambda_k)$ of $\Lambda$-modulus presheaves admits a symmetric monoidal structure $\otimes$ which extends $\otimes$ on $\ulMCor^\Lambda_k$ by colimits.
The functor $\tau\colon \MCor^\Lambda_k\to \ulMCor^\Lambda_k$ induces an adjunction
$$
	\tau_! \colon \PSh(\MCor^\Lambda_k)\rightleftarrows \PSh(\ulMCor^\Lambda_k) \colon \tau^*
$$
where $\tau^*$ is the restriction functor and $\tau_!$ is the left Kan extension of $\tau$.
Similarly, the functor $\omega\colon \ulMCor^\Lambda_k\to \Cor_k$ induces an adjunction
$$
	\underline{\omega}_! \colon \PSh(\ulMCor^\Lambda_k)\rightleftarrows \PSh(\Cor_k) \colon \underline{\omega}^*
$$
where $\underline{\omega}^*F(\mathcal{X})=F(X^\circ)$ and $\underline{\omega}_!F(X) = F(X,\emptyset)$.
We write $\omega_!:=\underline{\omega}_!\tau_!$ and $\omega^* = \tau^*\underline{\omega}^*$:
$$
	\omega_! \colon \PSh(\MCor^\Lambda_k)\rightleftarrows \PSh(\Cor_k) \colon \omega^*.
$$

In order to write down $\tau_!$ and $\omega_!$ explicitly, we use the notion of \emph{compactification}.
A compactification of $\mathcal{X}\in \ulMCor^\Lambda_k$ is a triple $(\overline{X},\overline{D},\Sigma)$ where $\overline{X}$ is a proper $k$-scheme and $\overline{D},\Sigma$ are $\Lambda$-effective $\Lambda$-Cartier divisors on $\overline{X}$, equipped with an isomorphism $X\simeq \overline{X}\setminus |\Sigma|$ over $k$ such that $D_X=\overline{D}|_X$.
We say that a compactification $(\overline{X}_1,\overline{D}_1,\Sigma_1)$ \emph{dominates} $(\overline{X}_2,\overline{D}_2,\Sigma_2)$ if the diagonal $\Delta_{X^\circ}\subset X^\circ\times X^\circ$ gives a morphism $(\overline{X}_1,\overline{D}_1+\Sigma_1)\to (\overline{X}_2, \overline{D}_2+\Sigma_2)$ in $\MCor^\Lambda_k$.
This defines a poset $\Comp^\Lambda(\mathcal{X})$ of compactifications of $\mathcal{X}$.

\begin{lemma}\label{Comp_is_cofiltered}\label{cofinal_in_Comp}\label{cor:pre_embedding}
	Let $\mathcal{X},\mathcal{Y}\in \ulMCor^\Lambda_k$.
	\begin{enumerate}
	\item	The poset $\Comp^\Lambda(\mathcal{X})$ is cofiltered.
	\item	If $(\overline{X},\overline{D},\Sigma)\in \Comp^\Lambda(\mathcal{X})$, then $\{(\overline{X},\overline{D},n\Sigma)\}_{n>0}$ is cofinal in $\Comp^\Lambda(\mathcal{X})$.
	\item	For any $\alpha\in \ulMCor^\Lambda_k(\mathcal{X},\mathcal{Y})$ and $(\overline{Y},\overline{E},\Xi)\in \Comp^\Lambda(\mathcal{Y})$, there exists $(\overline{X},\overline{D},\Sigma)\in \Comp^\Lambda(\mathcal{X})$ such that 
	$$
		\alpha\in \MCor^\Lambda_k((\overline{X},\overline{D}+\Sigma),(\overline{Y},\overline{E}+\Xi)).
	$$
	\end{enumerate}
\end{lemma}

\begin{proof}
	For $\Lambda = \mathbb{Z}$, this is proved in \cite[Section 1.8]{KMSY1}.
	For $\Lambda = \mathbb{Q}$, we may reduce to the case $\Lambda = \mathbb{Z}$ as in the proof of Lemma \ref{lem:composition_admissible}.
\end{proof}

By Lemma \ref{cor:pre_embedding}, we get a fully faithful symmetric monoidal functor
$$
	\ulMCor^\Lambda_k\to \Pro(\MCor^\Lambda_k);\quad \mathcal{X}\mapsto \quotprojlim_{(\overline{X},\overline{D},\Sigma)\in \Comp^\Lambda(\mathcal{X})}(\overline{X},\overline{D}+\Sigma)
$$
which is pro-left adjoint to the inclusion $\tau\colon \MCor^\Lambda_k\to \ulMCor^\Lambda_k$.
Therefore the functor $\tau_!$ can be written explicitly as
$$
	\tau_! F(\mathcal{X}) = \varinjlim_{(\overline{X},\overline{D},\Sigma)\in \Comp^\Lambda(\mathcal{X})} F(\overline{X},\overline{D}+\Sigma).
$$
For $X\in \Sm_k$, we can choose $\mathcal{Y}\in \MCor^\Lambda_k$ with $Y^\circ\simeq X$.
Then $\{(Y,\emptyset,nD_Y)\}_{n>0}$ is cofinal in $\Comp^\Lambda(X,\emptyset)$, so we get
$$
	\omega_!F(X) = \varinjlim_{n>0}F(Y,nD_Y).
$$

\begin{proposition}\label{omega_star_and_shriek}
	The following assertions hold:
	\begin{enumerate}
		\item	The functor $\omega^*$ is exact and fully faithful.
		\item	The functor $\omega_!$ is exact and symmetric monoidal.
	\end{enumerate}
\end{proposition}

\begin{proof}
	Both functors are clearly exact.
	Moreover, the above formula for $\omega_!$ implies that $\omega_!\omega^*\simeq\id$ and hence $\omega^*$ is fully faithful.
	Since $\omega$ is symmetric monoidal and $\omega_!$ is its extension by colimits, $\omega_!$ is also symmetric monoidal.
\end{proof}

\section{Cube invariance and reciprocity}\label{sec:ci-rec}
Fix $\Lambda\in \{\mathbb{Z},\mathbb{Q}\}$.
Following \cite{KSY2}, we introduce a class of $\Lambda$-modulus presheaves called \emph{cube invariant presheaves} which is an analogue of the class of $\mathbb{A}^1$-invariant presheaves used in the classical theory of motives.
This leads to the notion of \emph{$\Lambda$-reciprocity presheaf}, which is a presheaf with transfers that can be ``lifted'' to a cube invariant presheaf.
We show that the notion of $\mathbb{Q}$-reciprocity presheaf is actually the same as the notion of ($\mathbb{Z}$-)reciprocity presheaf.

\subsection{Cube invariance}

The object $\bcube:=(\mathbb{P}^1,[\infty])\in \MCor^\Lambda_k$ is called the {\it cube} over $k$.
We write $\pi\colon \bcube\to (\Spec k, \emptyset)$ for the ambient morphism given by the canonical projection $\mathbb{P}^1\to \Spec k$, and $i_\varepsilon\colon (\Spec k, \emptyset)\to \bcube\;(\varepsilon = 0,1)$ for the ambient morphism given by $\varepsilon\colon \Spec k\to \mathbb{P}^1$.

Let $\mathcal{X},\mathcal{Y}\in \MCor^\Lambda_k$.
Two morphisms $\alpha_0,\alpha_1\in \MCor^\Lambda_k(\mathcal{X}, \mathcal{Y})$ are called {\it cube homotopic} if there is some $\gamma\in \MCor^\Lambda_k(\mathcal{X}\otimes \bcube, \mathcal{Y})$ such that $\gamma\circ (\id\otimes i_\varepsilon)=\alpha_\varepsilon\;(\varepsilon = 0,1)$.
In this case we write $\alpha_0\sim \alpha_1$ and call $\gamma$ a \emph{cube homotopy} between $\alpha_0$ and $\alpha_1$.
We say that $\alpha\in \MCor^\Lambda_k(\mathcal{X}, \mathcal{Y})$ is a {\it cube homotopy equivalence} if there is some $\beta\in \MCor^\Lambda_k(\mathcal{Y},\mathcal{X})$ such that $\beta\circ \alpha\sim \id_{\mathcal{X}}$ and $\alpha\circ \beta\sim \id_{\mathcal{Y}}$ hold.
We call such $\beta$ a \emph{cube homotopy inverse} of $\alpha$.

\begin{lemma}
	The relation $\sim$ is an equivalence relation.
\end{lemma}

\begin{proof}
	Let $\mathcal{X},\mathcal{Y}\in \MCor^\Lambda_k$ and $\alpha_i\in \MCor^\Lambda_k(\mathcal{X},\mathcal{Y})\;(i=0,1,2)$.
	Let $\gamma_{01}\in \MCor^\Lambda_k(\mathcal{X}\otimes \bcube, \mathcal{Y})$ (resp. $\gamma_{12}\in \MCor^\Lambda_k(\mathcal{X}\otimes \bcube, \mathcal{Y})$) be a cube homotopy between $\alpha_0$ and $\alpha_1$ (resp. $\alpha_1$ and $\alpha_2$).
	Then
		$\gamma_{01}+\gamma_{12}-\alpha_1\circ (\id\otimes \pi)\colon \mathcal{X}\otimes \bcube\to \mathcal{Y}$
	gives a cube homotopy between $\alpha_0$ and $\alpha_2$.
\end{proof}

\begin{lemma}\label{cube_multiplication_admissible}
	Consider the multiplication map
	$\mu\colon \mathbb{A}^1\times \mathbb{A}^1\to \mathbb{A}^1;\;(x,y)\mapsto xy$.
	Then the graph of $\mu$ gives an element of $\MCor^\mathbb{Z}_k(\bcube\otimes\bcube,\bcube) = \MCor^\mathbb{Q}_k(\bcube\otimes\bcube,\bcube)$.
\end{lemma}

\begin{proof}
	See \cite[Lemma 5.1.1]{KMSY3}. 
\end{proof}

\begin{lemma}\label{cube_equivalence}
	For any $\mathcal{X}\in \MCor^\Lambda_k$, the modulus correspondence
	$\id\otimes \pi \in \MCor^\Lambda_k(\mathcal{X}\otimes \bcube, \mathcal{X})$ is a cube homotopy equivalence.
\end{lemma}

\begin{proof}
	Let us prove that $\id\otimes i_0\colon \mathcal{X}\to \mathcal{X}\otimes\bcube$ gives a cube homotopy inverse.
	The composition $(\id\otimes \pi)\circ (\id\otimes i_0)$ is the identity.
	Set
	$\gamma:=\id\otimes \mu\colon \mathcal{X}\otimes \bcube\otimes \bcube \to \mathcal{X}\otimes \bcube$.
	Then $\gamma\circ (\id\otimes i_0)=(\id\otimes i_0)\circ (\id\otimes \pi)$ and $\gamma\circ (\id\otimes i_1)=\id$, so we have $(\id\otimes i_0)\circ (\id\otimes \pi)\sim \id$.
\end{proof}

\begin{definition}
We say that $F\in \PSh(\MCor^\Lambda_k)$ is {\it cube invariant} if the map
$(\id \otimes \pi)^*\colon F(\mathcal{X})\to F(\mathcal{X}\otimes \bcube)$
is an isomorphism for all $\mathcal{X}\in \MCor^\Lambda_k$.
\end{definition}

\begin{lemma}\label{cube_invariance_characterization}
	For $F\in \PSh(\MCor^\Lambda_k)$, the following conditions are equivalent:
	\begin{enumerate}
		\item $F$ is cube invariant.
		\item For any $\mathcal{X}\in \MCor^\Lambda_k$, the map $(\id \otimes i_0)^*\colon F(\mathcal{X}\otimes \bcube)\to F(\mathcal{X})$ is injective.
		\item For any $\mathcal{X}\in \MCor^\Lambda_k$, the map $(\id \otimes i_0)^*-(\id\otimes i_1)^*\colon F(\mathcal{X}\otimes \bcube)\to F(\mathcal{X})$ is $0$.
		\item For any two cube homotopic morphisms $\alpha_0,\alpha_1$ in $\MCor^\Lambda_k$, we have $\alpha_0^*=\alpha_1^*$ on $F$.
	\end{enumerate}
\end{lemma}

\begin{proof}
	(1) $\iff$ (2) follows from $\pi\circ i_0=\id$.
	(1)$\implies$(3) follows from $\pi\circ i_0=\pi\circ i_1 = \id$.
	Let us prove (3)$\implies$(4).
	Let $\gamma\colon \mathcal{X}\otimes \bcube\to \mathcal{Y}$ be a cube homotopy between $\alpha_0$ and $\alpha_1$.
	Then we have $\alpha_0^*=(\id\otimes i_0)^*\gamma^*=(\id\otimes i_1)^*\gamma^*=\alpha_1^*$ on $F$.
	This shows that (3)$\implies$(4).
	If (4) holds, then for any cube homotopy equivalence $\alpha\colon \mathcal{X}\to\mathcal{Y}$ in $\MCor^\Lambda_k$, the map $\alpha^*\colon F(\mathcal{Y})\to F(\mathcal{X})$ is an isomorphism.
	Therefore (1) follows from Lemma \ref{cube_equivalence}.
\end{proof}

\begin{lemma}
	The class of cube invariant objects in $\PSh(\MCor^\Lambda_k)$ is closed under taking subobjects, quotients and extensions.
\end{lemma}

\begin{proof}
	The claim for subobjects follows from the equivalence of (1) and (2) in Lemma \ref{cube_invariance_characterization}.
	The remaining assertions then follow by the five lemma.
\end{proof}

There are two canonical ways to make a presheaf on $\MCor^\Lambda_k$ cube invariant: one is to take the maximal cube invariant quotient, and the other is to take the maximal cube invariant subobject.
The former is called the \emph{cube-localization} and the latter is called the \emph{cube invariant part}.

\begin{definition}
	The \emph{cube-localization} of $F\in \PSh(\MCor^\Lambda_k)$ is defined by
	$$
		\Lcube F (\mathcal{X}) := \Coker (F(\mathcal{X}\otimes \bcube)\xrightarrow{i_0^*-i_1^*}F(\mathcal{X})).
	$$
	There is a canonical epimorphism $F\twoheadrightarrow \Lcube F$.
	We write $\Lcube(\mathcal{X})$ for $\Lcube \mathbb{Z}_\tr(\mathcal{X})$.
\end{definition}

\begin{lemma}\label{properties_of_lower_h0}
	The following assertions hold.
	\begin{enumerate}
		\item For any $F\in \PSh(\MCor^\Lambda_k)$, the quotient $\Lcube F$ of $F$ is cube invariant.
		\item Let $F,G\in \PSh(\MCor^\Lambda_k)$.
		If $G$ is cube invariant, then the canonical homomorphism
		$\Hom(\Lcube F,G)\to \Hom(F,G)$
		is an isomorphism.
		In other words, $\Lcube F$ is the maximal cube invariant quotient of $F$.
		\item For any $\mathcal{X}\in \MCor^\Lambda_k$, the morphism $\Lcube(\mathcal{X}\otimes \bcube)\to \Lcube (\mathcal{X})$ induced by $\id\otimes \pi$ is an isomorphism.
	\end{enumerate}
\end{lemma}

\begin{proof}
		(1) follows from the equivalence of (1) and (3) in Lemma \ref{cube_invariance_characterization}.
		(2)	The canonical morphism $G\to \Lcube G$ is an isomorphism by Lemma \ref{cube_invariance_characterization}.
		The claim follows from this and (1).	
		(3)	By the Yoneda lemma, it suffices to prove that if $F\in \PSh(\MCor^\Lambda_k)$ is cube invariant then
				$$
					\Hom(\Lcube(\mathcal{X}),F)\to \Hom(\Lcube(\mathcal{X}\otimes \bcube),F)
				$$
				is an isomorphism.
				By (2), this map can be identified with $(\id\otimes \pi)^*\colon F(\mathcal{X})\to F(\mathcal{X}\otimes \bcube)$,
				which is an isomorphism since $F$ is cube invariant.
\end{proof}

\begin{definition}
	The \emph{cube invariant part} of $F\in \PSh(\MCor^\Lambda_k)$ is defined by
	$$
		\Rcube  F (\mathcal{X}) := \Hom(\Lcube(\mathcal{X}), F).
	$$
\end{definition}

\begin{lemma}\label{properties_of_upper_h0}
	The following assertions hold.
	\begin{enumerate}
		\item	For any $F\in \PSh(\MCor^\Lambda_k)$, the subobject $\Rcube F$ of $F$ is cube invariant.
		\item	Let $F,G\in \PSh(\MCor^\Lambda_k)$.
		If $F$ is cube invariant, then the canonical homomorphism
		$\Hom(F,\Rcube G)\to \Hom(F,G)$
		is an isomorphism.
		In other words, $\Rcube F$ is the maximal cube invariant subobject of $F$.
		\item	The functor $\Rcube\colon \PSh(\MCor^\Lambda_k)\to \PSh(\MCor^\Lambda_k)$ is right adjoint to $\Lcube$.
	\end{enumerate}
\end{lemma}

\begin{proof}
	(1) follows from Lemma \ref{properties_of_lower_h0} (3).
	(2) If $F\in \PSh(\MCor^\Lambda_k)$ is cube invariant, then the canonical morphism $\Rcube F\to F$ is an isomorphism.
		The claim follows from this and (1).	
	(3) follows from (2) and Lemma \ref{properties_of_lower_h0} (2).
\end{proof}

\begin{remark}
	In \cite{KSY2}, $\Lcube$ and $\Rcube$ are defined to be functors taking values in the category of cube invariant presheaves.
	On the other hand, we define $\Lcube$ and $\Rcube$ as \emph{endofunctors} of $\PSh(\MCor^\Lambda_k)$.
	This has the advantage that $\Rcube$ becomes right adjoint to $\Lcube$.
\end{remark}

\begin{lemma}\label{lem:h0_tensor}
	Let $F,G\in \PSh(\MCor^\Lambda_k)$.
	\begin{enumerate}
		\item	If $G$ is cube invariant, then $\intHom(F,G)$ is also cube invariant.
		\item	The canonical epimorphism $\Lcube(F\otimes G)\twoheadrightarrow \Lcube((\Lcube F)\otimes G)$ is an isomorphism.
	\end{enumerate}
\end{lemma}

\begin{proof}
		(1)	For any $\mathcal{X}\in \MCor^\Lambda_k$, we have
				$\intHom(F,G)(\mathcal{X}) =\Hom(\mathbb{Z}_\tr(\mathcal{X})\otimes F, G)$.
				Since $G$ is cube invariant, $\intHom(F,G)$ is cube invariant when $F$ is representable.
				Since $\intHom({-},G)$ turns colimits into limits, the same is true for a general $F$.
				
		(2)	For any $H\in \PSh(\MCor^\Lambda_k)$, we have
				\begin{align*}
					\Hom(\Lcube(F\otimes G),H)
					\simeq{}&\Hom(F,\intHom(G,\Rcube H))&\text{($\Lcube\dashv \Rcube$)}\\
					\simeq{}&\Hom(\Lcube F, \intHom(G,\Rcube H))&\text{(by (1))}\\
					\simeq{}&\Hom(\Lcube((\Lcube F)\otimes G),H).&\text{($\Lcube\dashv \Rcube$)}
				\end{align*}
				Therefore we get the desired result by the Yoneda lemma.\qedhere
\end{proof}

\begin{definition}\label{def:Lh0}
	We define two functors $h_0$ and $h^0$ as follows.
	\begin{enumerate}
	\item	We define $h_0 := \omega_! \Lcube \colon \PSh(\MCor^\Lambda_k)\to \PSh(\Cor_k)$.
	Since $\Lcube$ is lax symmetric monoidal and $\omega_!$ is symmetric monoidal, it follows that $h_0$ is lax symmetric monoidal.
	We write $h_0(\mathcal{X})$ for $h_0\mathbb{Z}_\tr(\mathcal{X})$.
	\item	We define $h^0:=\Rcube\omega^*\colon \PSh(\Cor_k)\to \PSh(\MCor^\Lambda_k)$.
	It follows from the adjunctions $\Lcube\dashv \Rcube$ and $\omega_!\dashv \omega^*$ that $h^0$ is right adjoint to $h_0$.
	\end{enumerate}
\end{definition}

Note that, by applying $\omega_!$ to the canonical epimorphism $\mathbb{Z}_\tr(\mathcal{X})\twoheadrightarrow \Lcube (\mathcal{X})$, we obtain a canonical epimorphism $\mathbb{Z}_\tr(X^\circ)\twoheadrightarrow h_0 (\mathcal{X})$.

\begin{remark}\label{h0_lax_monoidal}
	By Lemma \ref{lem:h0_tensor}, we have $\Lcube(F\otimes G)\simeq \Lcube((\Lcube F)\otimes (\Lcube G))$ for any $F,G\in \PSh(\MCor^\Lambda_k)$.
	The latter is a quotient of $\Lcube F\otimes \Lcube G$, so we get an epimorphism
	$\Lcube F\otimes \Lcube G \twoheadrightarrow \Lcube(F\otimes G)$.
	In particular, $\Lcube$ admits a canonical lax symmetric monoidal structure.
	Applying the exact symmetric monoidal functor $\omega_!$, we also see that there is a canonical epimorphism $h_0 F\otimes h_0 G\twoheadrightarrow h_0(F\otimes G)$.
\end{remark}

\begin{remark}
	In \cite{KSY2}, a functor named $\omega^{\CI}$ is defined by the same formula as $h^0$, but it is regarded as a functor taking values in the category of cube invariant presheaves rather than $\PSh(\MCor^\Lambda_k)$.
\end{remark}

\begin{lemma}\label{Lh0_description}
	Let $\mathcal{X}\in \MCor^\Lambda_k$ and $Y\in \Sm_k$.
	The canonical epimorphism $\mathbb{Z}_\tr(X^\circ)\twoheadrightarrow h_0(\mathcal{X})$ induces an isomorphism
	$h_0(\mathcal{X})(Y) \simeq \Coker (\ulMCor^\Lambda_k(Y\otimes \bcube, \mathcal{X})\xrightarrow{i_0^*-i_1^*} \Cor_k(Y,X^\circ))$.
\end{lemma}

\begin{proof}
	This follows from the definition of $h_0$ and the isomorphism $\mathbb{Z}_\tr(\mathcal{X})(Y)\simeq \mathbb{Z}_\tr(X^\circ)(Y)$.
\end{proof}

The functors we have defined so far can be summarized as follows.
$$
\xymatrix{
	\PSh(\MCor^\Lambda_k)
		\ar@<1ex>[r]^-{\Lcube}
		\ar@/^23pt/[rr]^-{h_0}
	&
	\PSh(\MCor^\Lambda_k)
		\ar@<1ex>[r]^-{\omega_!}
		\ar@<1ex>[l]^-{\Rcube}
		\ar@{}[l]|-{\text{\rotatebox{-90}{$\dashv$}}}
	&
	\PSh(\Cor_k)
		\ar@<1ex>[l]^-{\omega^*}	
		\ar@{}[l]|-{\text{\rotatebox{-90}{$\dashv$}}}
		\ar@/^23pt/[ll]^-{h^0}
}
$$

\subsection{Reciprocity sheaves}

We define the notion of $\Lambda$-reciprocity sheaf following \cite{KSY2}.

\begin{definition}
Let $X\in \Sm_k$, $F\in \PSh(\Cor_k)$ and $a\in F(X)$.
Take $\mathcal{Y}\in \MCor^\Lambda_k$ with $Y^\circ \simeq X$.
We say that $\mathcal{Y}$ is a {\it $\Lambda$-modulus} for $a$ if $a\colon \mathbb{Z}_\tr(X)\to F$ factors through the canonical epimorphism
$\mathbb{Z}_\tr(X)\twoheadrightarrow h_0(\mathcal{Y})$.
This is equivalent to saying that $a\colon \mathbb{Z}_\tr(\mathcal{Y})\to \omega^*F$ factors through $h^0F\subset \omega^*F$.
\end{definition}

\begin{lemma}\label{reciprocity_equivalent}
	Let $F\in \PSh(\Cor_k)$.
	The following conditions are equivalent:
	\begin{enumerate}
		\item	The counit morphism $h_0h^0F\to F$ is an isomorphism.
		\item	For any $X\in \Sm_k$ and any $a\in F(X)$, there exists $\mathcal{Y}\in \MCor^\Lambda_k$ with $Y^\circ \simeq X$ such that
				$a\in h^0F(\mathcal{Y})$.
		\item	Every section of $F$ admits a $\Lambda$-modulus.
	\end{enumerate}
\end{lemma}

\begin{proof}
	The equivalence of (1) and (2) follows from the formula
	$$
		h_0h^0F (X)=\omega_!h^0F(X)=\varinjlim_{\mathcal{Y}\in \MCor^\Lambda_k,Y^\circ\simeq X} h^0F(\mathcal{Y}) \subset F(X).
	$$
	The equivalence of (2) and (3) is clear from the definition.
\end{proof}

\begin{definition}
We say that $F$ has {\it $\Lambda$-reciprocity} or $F$ is a \emph{$\Lambda$-reciprocity presheaf} if it satisfies the equivalent conditions in Lemma \ref{reciprocity_equivalent}.
If $F$ is moreover a Nisnevich sheaf, then we say that $F$ is a \emph{reciprocity sheaf}.
We define $\RSC^\Lambda_k$ (resp. $\RSC^\Lambda_{k,\Nis}$) to be the full subcategory of $\PSh(\Cor_k)$ spanned by $\Lambda$-reciprocity presheaves (resp. $\Lambda$-reciprocity sheaves).
\end{definition}

The category $\RSC^\Lambda_k$ is closed under taking subobjects and quotients in $\PSh(\Cor_k)$.
In particular, $\RSC^\Lambda_k$ is an abelian category and the inclusion functor $\RSC^\Lambda_k\to \PSh(\Cor_k)$ is exact.

\begin{lemma}
	We have $\RSC^\mathbb{Z}_k=\RSC^\mathbb{Q}_k$.
\end{lemma}

\begin{proof}
	Let $X\in \Sm_k$, $F\in \PSh(\Cor_k)$ and $a\in F(X)$.
	Then it is clear that every $\mathbb{Z}$-modulus $\mathcal{Y}$ for $a$ is also a $\mathbb{Q}$-modulus for $a$, since the description of $h_0(\mathcal{Y})$ (Lemma \ref{Lh0_description}) is the same for $\Lambda = \mathbb{Z},\mathbb{Q}$.
	Conversely, if $\mathcal{Y}$ is a $\mathbb{Q}$-modulus for $a$, then so is $(Y,nD_Y)$ for $n\in \mathbb{Z}_{>0}$, so there exists a $\mathbb{Z}$-modulus for $a$.
\end{proof}

We thus write $\RSC_k$ for $\RSC^\mathbb{Z}_k=\RSC^\mathbb{Q}_k$ and call its objects \emph{reciprocity presheaves} over $k$.

\begin{lemma}\label{omega_shriek_CI_is_RSC}
	If $F\in \PSh(\MCor^\Lambda_k)$ is cube invariant, then $\omega_!F\in \RSC_k$.
\end{lemma}

\begin{proof}
	The composition $h_0F\xrightarrow{h_0(\text{unit})} h_0h^0h_0F\xrightarrow{\text{counit}} h_0F$ is the identity.
	Since the counit morphism $h_0h^0G\to G$ is a monomorphism for any $G\in \PSh(\Cor_S)$, it follows that the counit morphism $h_0h^0h_0F\to h_0F$ is an isomorphism.
	Therefore $\omega_!F = h_0F$ has reciprocity.
\end{proof}

\begin{remark}
	Let $F\in \RSC_k$ and $\mathcal{X}\in \MCor^\Lambda_k$.
	The group $h^0 F(\mathcal{X})$ can be thought of as the subgroup of $\omega^*F(\mathcal{X})=F(X^\circ)$ consisting of elements whose ``ramification'' is bounded by $D_X$.
	This actually recovers several classical notions in ramification theory such as the Artin conductor or the irregularity \cite{RS21}.
	For further developments of the ramification theory of reciprocity sheaves, see \cite{RS21}, \cite{RSram1}, \cite{RSram2}, and \cite{RSram3}.
\end{remark}

\section{Modulus curves and motivic Hasse-Arf theorem}\label{sec:HA}

Fix $\Lambda\in \{\mathbb{Z},\mathbb{Q}\}$.
In this section, we describe $h_0(\mathcal{X})$ for a $\Lambda$-modulus curve (Definition \ref{def:modulus_curve}) $\mathcal{X}$ over $k$, using the Chow group of relative $0$-cycles (Definition \ref{def:CH0}).
As a corollary, we prove a motivic analogue of the Hasse-Arf theorem.

\subsection{Admissible rational functions}

\begin{definition}
	A $\Lambda$-modulus pair $\mathcal{X}\in \ulMCor^\Lambda_k$ is called \emph{normal integral} if $X$ is a normal integral scheme.
\end{definition}

\begin{definition}\label{def:admissible_functions}
	Let $\mathcal{X}\in \ulMCor^\Lambda_k$ be a normal integral $\Lambda$-modulus pair.
	We say that $f\in k(X)^\times$ is \emph{admissible} with respect to $D_X$ if there is some open neighborhood $U$ of $|D_X|$ such that 
	\begin{enumerate}
		\item	$f$ is regular and invertible on $U$, and
		\item	$\div(f-1)\geq D_X$ holds on $U$ (including the case $f=1$).
	\end{enumerate}
	We define $\Adm(\mathcal{X})\subset k(X)^\times$ to be the subset consisting of rational functions which are admissible with respect to $D_X$.
	It follows from the identity $fg-1=(f-1)g+(g-1)$ that $\Adm(\mathcal{X})$ is a subgroup of $k(X)^\times$.
\end{definition}

\begin{lemma}\label{lem:G_valuative}
	Let $\mathcal{X}\in \ulMCor^\Lambda_k$ be a normal integral $\Lambda$-modulus pair.
	Then $f\in k(X)^\times$ is admissible with respect to $D_X$ if and only if for any discrete valuation ring $R$ with fraction field $k(X)$ and any morphism $\rho\colon \Spec R\to X$ extending $\Spec k(X)\to X$ such that $\rho^*D_X\neq\emptyset$, the inequality $v_R(f-1)\geq v_R(\rho^*D_X)$ holds in $\Lambda$.
\end{lemma}

\begin{proof}
	The ``only if'' part is easy.
	Let us prove the ``if'' part.
	Let $x\in |D_X|$.
	Our assumption implies that for any discrete valuation ring $R$ with fraction field $k(X)$ dominating $\mathcal{O}_{X,x}$, we have $f\in R^\times$.
	Since $\mathcal{O}_{X,x}$ is a noetherian normal domain, this shows that $f\in \mathcal{O}_{X,x}^\times$.
	Therefore $f$ is invertible on some open neighborhood $U$ of $|D_X|$.
	The condition (1) in Definition \ref{def:admissible_functions} is satisfied for this $U$.
	Let us verify the condition (2).
	Let $R$ be a valuation ring with fraction field $k(X)$ and $\rho\colon \Spec R\to U$ be a morphism extending $\Spec k(X)\to U$.
	If $\rho^*D_X$ is trivial, then we have $\rho^*\div(f-1)\geq 0=\rho^*D_X$, since $f-1$ is regular on $U$.
	Otherwise, our assumption implies that $\rho^*\div(f-1)\geq \rho^*D_X$.
	Since $U$ is normal, we get $\div(f-1)\geq D_X$ on $U$.
\end{proof}

\begin{lemma}\label{higher_unit_norm}
	Let $(K,v)$ be a discrete valuation field and $L/K$ be a finite extension.
	Let $\gamma\in \mathbb{Q}_{>0}$, $f\in L^\times$ and suppose that for every discrete valuation $w\colon L^\times\to \mathbb{Q}$ extending $v$, we have
	$w(f-1)\geq \gamma$.
	Then we have
	$v(\Nm_{L/K}(f) - 1) \geq \gamma$.
\end{lemma}

\begin{proof}
	Let $i$ denote the inseparable degree of $L/K$.
	Fix an algebraic closure $\overline{K}$ of $K$.
	Then $v$ can be extended to a valuation on $\overline{K}$.
	Let us fix such an extension $\widetilde{v}$.
	We have
	$$
		\Nm_{L/K}(f)=\textstyle \prod_{\sigma} \sigma(f)^i =\prod_{\sigma} (1+\sigma(f-1))^i
	$$
	where $\sigma$ runs over the set of distinct $K$-embeddings of $L$ into $\overline{K}$.
	Now $\widetilde{v}\circ \sigma$ gives a valuation on $L$ extending $v$ for each $\sigma$, so we have
	$\widetilde{v}(\sigma(f-1))\geq \gamma$ by our assumption.
	By the above expression for $\Nm_{L/K}(f)$, we get the desired inequality.
\end{proof}

\begin{lemma}\label{principal_divisor_push}
	Let $\mathcal{X}, \mathcal{C} \in \MCor^\Lambda_k$ be normal integral $\Lambda$-modulus pairs and $q\colon \mathcal{X}\to \mathcal{C}$ be an ambient morphism such that $X \to C$ is proper and generically finite.
	Then for any $f\in \Adm(\mathcal{X})$, we have $\Nm_{k(X)/k(C)}(f)\in \Adm(\mathcal{C})$.
\end{lemma}

\begin{proof}
	By Lemma \ref{lem:G_valuative}, it suffices to show that for any discrete valuation ring $R$ with fraction field $k(C)$ and any morphism $\rho\colon \Spec R\to C$ extending $\Spec k(C)\to C$ such that $\rho^*D_C\neq\emptyset$, we have $v_R(\Nm_{k(X)/k(C)}(f)-1)\geq v_R(\rho^*D_C)$.
	Let $w\colon k(X)^\times \to \mathbb{Q}$ be an arbitrary extension of $v_R$ and $R'$ be its valuation ring.
	Since $q\colon X\to C$ is proper, there is a unique dashed arrow $\varphi$ in the following diagram which makes it commute:
	$$
		\xymatrix{
			\Spec k(X)\ar[r]\ar[d]		&\Spec R'\ar@{-->}[r]^-{\varphi}\ar[d]		&X\ar[d]^-q\\
			\Spec k(C)\ar[r]			&\Spec R\ar[r]^-{\rho}						&C.
		}
	$$
	By Lemma \ref{lem:G_valuative}, we have $w(f-1)\geq w(\varphi^*D_X)\geq w(\varphi^*q^*D_C)=v_R(\rho^*D_C)$.
	By Lemma \ref{higher_unit_norm}, we get $v_R(\Nm_{k(X)/k(C)}(f)-1)\geq v_R(\rho^*D_C)$, as was to be shown.
\end{proof}

\subsection{Modulus curves}

In this subsection, we fix $S\in \Sm_k$ which is connected.

\begin{definition}\label{def:modulus_curve}
	A {\it $\Lambda$-modulus curve} over $S$ is a modulus pair $\mathcal{C}\in \ulMCor^\Lambda_k$ equipped with a proper smooth morphism $C\to S$ of relative dimension $1$ such that $C^\circ \to S$ is quasi-affine.
\end{definition}

\begin{lemma}\label{modulus_curve_every_fiber}
	Let $\mathcal{C}$ be a $\Lambda$-modulus curve over $S$.
	If $V$ is a closed subscheme of $C$ contained in $C^\circ$, then $V$ is finite over $S$.
\end{lemma}

\begin{proof}
	Since $C\to S$ is proper, $V\to S$ is also proper.
	For any $x\in S$, the fiber $V_x$ of $V$ is a closed subscheme of $C^\circ_x$ which is proper over $\Spec k(x)$.
	By our assumption that $C^\circ_x$ is quasi-affine, there is an open immersion $C^\circ_x\hookrightarrow B$ into an affine scheme $B$.
	The properness of $V_x$ implies that the image of $V_x$ in $B$ is closed.
	It follows that $V_x$ is finite over $k(x)$ and hence $V\to S$ is quasi-finite.
	By Zariski's main theorem, we conclude that $V\to S$ is finite.
\end{proof}

\begin{lemma}\label{lem:modulus_curve_Weil}
	Let $\mathcal{C}$ be a $\Lambda$-modulus curve over $S$.
	For an integral closed subscheme $V\subset C^\circ$, the following conditions are equivalent:
	\begin{enumerate}
	\item	$V\to S$ is finite surjective.
	\item	$V$ has codimension $1$ and is closed in $C$.
	\end{enumerate}
\end{lemma}

\begin{proof}
	(1)$\implies$(2): Suppose that $V\to S$ is finite surjective.
	Then $V$ is finite over $C$ and hence is closed in $C$.
	Let $\xi$ be the generic point of $V$ and $\eta$ its image in $S$.
	Then the fiber dimension theorem for flat morphisms \cite[Corollary 14.95]{Goertz-Wedhorn} implies
	$\codim_{C^\circ} (\xi) = \codim_S (\eta) + \codim_{C^\circ_\eta} (\xi)$.
	Since $V\to S$ is finite surjective, we have $\codim_S (\eta) = 0$ and $\codim_{C^\circ_\eta} (\xi)=1$.
	This implies $\codim_{C^\circ} (\xi)=1$.
	
	(2)$\implies$(1): Suppose that $V$ has codimension $1$ and is closed in $C$.
	Then $V$ is finite over $S$ by Lemma \ref{modulus_curve_every_fiber}.
	We have $\codim_{C^\circ} (\xi) = 1$ and $\codim_{C^\circ_\eta} (\xi)=1$, so the fiber dimension theorem implies $\codim_S(\eta) = 0$.
	Therefore $V\to S$ is surjective.
\end{proof}

\begin{definition}\label{def:CH0}
	Let $\mathcal{C}$ be a $\Lambda$-modulus curve over $S$ with $C$ connected.
	Then for any $f\in \Adm(\mathcal{C})$, the components of the Weil divisor $\div(f)$ on $C^\circ$ are closed in $C$ and hence $\div(f)\in \Cor_S(S,C^\circ)$ by Lemma \ref{lem:modulus_curve_Weil}.
	We define the \emph{Chow group of relative 0-cycles} of $\mathcal{C}$ by
	$$
		\CH_0(\mathcal{C}/S):=\Coker(\Adm(\mathcal{C})\xrightarrow{\div}\Cor_S(S,C^\circ)).
	$$
	If $\mathcal{C}$ is a $\Lambda$-modulus curve over $S$ with $C$ non-connected, then we define $\CH_0(\mathcal{C}/S)$ by taking a direct sum over connected components.
\end{definition}

\begin{remark}
	When $S=\Spec k$ and $\Lambda = \mathbb{Z}$, the above definition of $\CH_0(\mathcal{C}/S)$ coincides with the definition of the \emph{relative Chow group of $0$-cycles} $\mathrm{C}(C,D_C)$ from \cite[Definition 1.6]{Kerz-Saito}.
	The latter group is also defined for higher dimensional varieties, and it was used to establish ramified higher dimensional class field theory.
\end{remark}

When $\Lambda = \mathbb{Z}$, we have the following comparison result between the Chow group of relative 0-cycles and the relative Picard group.

\begin{lemma}\label{relative_Picard_comparison}
	Let $\mathcal{C}$ be a $\mathbb{Z}$-modulus curve over $S$ such that $D_C$ is contained in some affine open subset of $C$.
	Then we have $\CH_0(\mathcal{C}/S)\simeq \Pic(C,D_C)$.
\end{lemma}

\begin{proof}
	We may assume that $C$ is connected.
	Then $\Adm(\mathcal{C})$ consists of rational functions $f$ which is regular and invertible on some neighborhood of $D_C$ and $f|_{D_C}=1$.
	On the other hand, $\Pic(C,D_C)$ is by definition the group of isomorphism classes of pairs $(L,\alpha)$ where $L$ is a line bundle on $C$ and $\alpha$ is a nowhere-vanishing section of $L|_{D_C}$.
	By our assumption that $D_C$ has an affine open neighborhood, any such $\alpha$ can be extended to a rational section of $L$ which is regular and invertible on some neighborhood of $|D_C|$.
	For two such extensions $\widetilde{\alpha}_1,\widetilde{\alpha}_2$, the quotient $\widetilde{\alpha}_1/\widetilde{\alpha}_2$ is admissible with respect to $D_C$.
	Therefore $\Pic(C,D_C)$ can be identified with the cokernel of $\div\colon \Adm(\mathcal{C})\to Q$ where $Q$ is the group of Cartier divisors on $C$ whose support is disjoint from $|D_C|$.
	We have $Q = \Cor_S(S,C^\circ)$ by Lemma \ref{lem:modulus_curve_Weil}, so the claim follows.
\end{proof}

The following result generalizes \cite[Theorem 1.1]{RY16} to $\Lambda = \mathbb{Q}$.

\begin{lemma}\label{modulus_curve_suslin_homology_lemma}
	Let $\mathcal{C}$ be a $\Lambda$-modulus curve over $S$.
	Let $\ulMCor^\Lambda_S(\bcube\times S, \mathcal{C})$ be the intersection of $\ulMCor^\Lambda_k(\bcube\times S,\mathcal{C})$ and $\Cor_S(\mathbb{A}^1\times S, C^\circ)$.
	Then we have
	$$
		\CH_0(\mathcal{C}/S) \simeq \Coker(
		\ulMCor^\Lambda_S(\bcube\times S, \mathcal{C})
		\xrightarrow{i_0^*-i_1^*}
		\Cor_S(S,C^\circ)
		).
	$$
\end{lemma}

\begin{proof}
	We may assume that $C$ is connected.
	Throughout this proof we identify $\bcube$ with $(\mathbb{P}^1,[1])$ via the isomorphism $t\mapsto \frac{t}{t-1}$, and set $\square:=\mathbb{P}^1\setminus\{1\}$.
	It suffices to prove that
	$$
		\Im(\Adm(\mathcal{C})\xrightarrow{\div} \Cor_S(S,C^\circ))=
		\Im(\ulMCor^\Lambda_S(\bcube\times S, \mathcal{C})\xrightarrow{i_0^*-i_\infty^*}\Cor_S(S,C^\circ)).
	$$
	We will prove a stronger statement that there is a surjective homomorphism
	$\Phi\colon \ulMCor^\Lambda_S(\bcube\times S, \mathcal{C})\twoheadrightarrow \Adm(\mathcal{C})$
	which makes the following diagram commutative:
	\begin{align}\label{modulus_curve_diagram}
		\xymatrix{
			\ulMCor^\Lambda_S(\bcube\times S, \mathcal{C})
			\ar@{->>}[r]^-{\Phi}\ar[rd]	_-{i_0^*-i_\infty^*}	&\Adm(\mathcal{C})\ar[d]^-{\div}\\
																	&\Cor_S(S,C^\circ).
		}
	\end{align}
	
	Let $[V]\in \ulMCor^\Lambda_S(\bcube\times S, \mathcal{C})$.
	Let $\overline{V}$ denote the scheme-theoretic closure of $V$ in $\mathbb{P}^1\times S\times C$ and $p\colon \overline{V}\to \mathbb{P}^1$, $q\colon \overline{V}\to C$ be the canonical projections.
	Then we have $p^*[1]|_{\overline{V}^N}\geq q^*D_C|_{\overline{V}^N}$ and hence $p\in \Adm(\overline{V}^N,q^*D_C|_{\overline{V}^N})$.
	We define
	$$
		\Phi([V])=\begin{cases}
			\Nm_{k(V)/k(C)}(p)		&(V\to C\text{ is generically finite}),\\
			1								&(\text{otherwise}).
		\end{cases}
	$$
	Lemma \ref{principal_divisor_push} shows that this gives a homomorphism $\Phi\colon \ulMCor^\Lambda_S(\bcube\times S, \mathcal{C})\to \Adm(\mathcal{C})$.
	
	Let us prove that (\ref{modulus_curve_diagram}) is commutative.
	Let $K$ be the function field of $S$.
	Since the operations appearing in (\ref{modulus_curve_diagram}) are compatible with base change to $\Spec K$ and the map $\Cor_S(S,C^\circ)\to \Cor_K(\Spec K,C^\circ_K)$ is injective, we may assume that $S=\Spec K$.
	In this case the claim follows from a standard computation of cycles.
	
	It remains to show that $\Phi$ is surjective.
	Let $f\in \Adm(\mathcal{C})$.
	If $f=1$, then $\Phi(0)=f$.
	Otherwise, we take an open neighborhood $U$ of $|D_C|$ satisfying the conditions (1) and (2) in Definition \ref{def:admissible_functions}.
	Then $f$ is regular and invertible on $U$.
	Define $\overline{V}\subset \mathbb{P}^1\times S\times C$ to be the closure of the graph of $f|_U\colon U\to \mathbb{A}^1\times S$.
	Let $p\colon \overline{V}\to \mathbb{P}^1$ and $q\colon \overline{V}\to C$ be the canonical projections.
	Let $V\subset \overline{V}$ denote the inverse image of $\square=\mathbb{P}^1\setminus\{1\}$ under $p$.
	Since $f=1$ holds on $|D_C|$, $V$ is contained in $\square\times C^\circ$.
	We will prove that $[V]$ gives an element of $\ulMCor^\Lambda_S(\bcube\times S, \mathcal{C})$ such that $\Phi([V])=f$.
		
	Since $f=1$ holds on $|D_C|$ and $f\neq 1$, $V$ is dominant over $\square\times S$.
	Applying Lemma \ref{modulus_curve_every_fiber} to the modulus curve $\square\times \mathcal{C}$ over $\square\times S$, we see that $V$ is finite over $\square\times S$.
	The admissibility of $V$ follows from the assumption that $\div(f-1)\geq\pi^*D_C$ holds on $U$.
	Finally, we have $\Phi([V])=\Nm_{k(V)/k(C)}(p)=f$ by construction.
\end{proof}

\begin{theorem}\label{modulus_curve_suslin_homology}
	Let $\mathcal{C}$ be a $\Lambda$-modulus curve over $k$.
	Then for any connected $S\in \Sm_k$, the canonical surjection $\Cor_k(S,C^\circ)=\Cor_S(S,S\times C^\circ)\twoheadrightarrow \CH_0(S\times \mathcal{C}/S)$ induces an isomorphism
	$$h_0(\mathcal{C})(S)\xrightarrow{\sim} \CH_0(S\times \mathcal{C}/S).$$
	Moreover, if $\Lambda = \mathbb{Z}$ and $S$ is affine, then we also have
	$\CH_0(S\times \mathcal{C}/S)\simeq \Pic(S\times C,S\times D_C)$.
\end{theorem}

\begin{proof}
	Applying Lemma \ref{modulus_curve_suslin_homology_lemma} to the $\Lambda$-modulus curve $S\times\mathcal{C}$ over $S$, we get
	\begin{align*}
	\CH_0(\mathcal{C}/S) \simeq {} &\Coker(\ulMCor^\Lambda_S(\bcube_S, S\times \mathcal{C})\xrightarrow{i_0^*-i_1^*}\Cor_S(S,S\times C^\circ))\\
	\simeq {} &\Coker(\ulMCor^\Lambda_k(\bcube_S,\mathcal{C})\xrightarrow{i_0^*-i_1^*}\Cor_k(S,C^\circ)).
	\end{align*}
	The last term is isomorphic to $h_0(\mathcal{C})(S)$ by Lemma \ref{Lh0_description}.
	The second assertion follows from Lemma \ref{relative_Picard_comparison}.
\end{proof}

\subsection{Motivic Hasse-Arf theorem}

\begin{theorem}\label{geometric Hasse-Arf for curves}
	Let $\mathcal{C}$ be a $\mathbb{Q}$-modulus curve over $k$.
	Then we have
	$$
		h_0(C,\lceil D_C\rceil)\xrightarrow{\sim} h_0(C,D_C).
	$$
\end{theorem}

\begin{proof}
	We want to prove that for any connected $S\in \Sm_k$, the canonical map
	$h_0(C,\lceil D_C\rceil)(S)\to h_0(C,D_C)(S)$
	is an isomorphism.
	By Theorem \ref{modulus_curve_suslin_homology}, it suffices to show that
	$$
		\CH_0((S\times C,S\times \lceil D_C\rceil)/S)\to \CH_0((S\times C,S\times D_C)/S)
	$$
	is an isomorphism.
	We have $S\times\lceil D_C\rceil=\lceil S\times D_C\rceil$ since $S$ is smooth over $k$.
	Now the claim follows from the fact that $\div(f-1)$ has integral coefficients for any rational function $f$ on $S\times C$.
\end{proof}

The following corollary can be seen as a motivic analogue of the Hasse-Arf theorem.

\begin{corollary}\label{Motivic Hasse-Arf for curves}
	Let $\mathcal{C}$ be a $\mathbb{Q}$-modulus curve over $k$.
	Then for any $F\in \RSC_k$ we have
	$$
		h^0F(C,D_C)\xrightarrow{\sim} h^0F(C,\lceil D_C\rceil).
	$$
\end{corollary}

\begin{proof}
	This follows from Theorem \ref{geometric Hasse-Arf for curves}.
\end{proof}

\begin{remark}
	We have $\mathbb{Z}_\tr(X,D_X)\neq \mathbb{Z}_\tr(X,\lceil D_X\rceil)$ in general, so Corollary \ref{Motivic Hasse-Arf for curves} is not obvious from the definition.
	Actually, Corollary \ref{Motivic Hasse-Arf for curves} is false for modulus pairs of higher dimensions.
\end{remark}

\section{Motivic construction of the ring of Witt vectors}\label{sec:Witt}

In this section, we present a construction of the ring of Witt vectors using $\mathbb{Q}$-modulus pairs.

\subsection{Usual construction}
First we recall the usual definition of the ring of Witt vectors; see \cite{Hes15} for details.
Let $A$ be a ring and $n\geq 0$.
The \emph{group of big Witt vectors of length $n$} of $A$ is defined by
	$$\mathbb{W}_n(A):=1+tA[t]/(t^{n+1}) \subset (A[t]/(t^{n+1}))^\times.$$
For $a\in A$, we write $[a]$ for the element $1-at\in \mathbb{W}_n(A)$.
The presheaf $X\mapsto\mathbb{W}_n(\mathcal{O}(X))$ on $\Sm_k$ is represented by an algebraic group whose underlying scheme is isomorphic to $\mathbb{A}^n$, so it can be regarded as an object of $\PSh(\Cor_k)$ (see \cite[Lemma 1.4.4]{BVK16}).
There are several important morphisms in $\PSh(\Cor_k)$:
\begin{align*}
	U&\colon \mathbb{Z} \to \mathbb{W}_n,& \star&\colon \mathbb{W}_n\otimes\mathbb{W}_n\to \mathbb{W}_n,\\
	F_s&\colon \mathbb{W}_{sn}\to \mathbb{W}_n\quad (s\geq 1),& V_s&\colon \mathbb{W}_n\to \mathbb{W}_{sn}\quad (s\geq 1).
\end{align*}
They are called the unit, the multiplication, the Frobenius, and the Verschiebung.
The unit $U$ is given by $1\mapsto [1]$, and the Verschiebung $V_s$ is induced by $t\mapsto t^s$.
The multiplication $\star$ is characterized by $[a]\star[b]=[ab]$, and the Frobenius $F_s$ is characterized by $F_s([a])=[a^s]$.
We define
$$\widehat{\mathbb{W}}_n:=\mathbb{W}_n\oplus \mathbb{G}_m\oplus \mathbb{Z} \in \PSh(\Cor_k).$$
For a ring $A$ and $a\in A^\times$, we write $[a]$ for the element $([a],a,1)\in \widehat{\mathbb{W}}_n(A)$.
We extend the natural transformations $U$, $\star$, $F_s$, and $V_s$ to $\widehat{\mathbb{W}}_n$ by setting
\begin{align*}
	&U(1)=[1],\quad (\alpha,a,m)\star(\alpha',a',m')=(\alpha\star\alpha',a^{m'}{a'}^m,mm'),\\
	&F_s(\alpha,a,m)=(F_s(\alpha),a^s,m),\quad V_s(\alpha,a,m)=(V_s(\alpha),a,sm).
\end{align*}

\begin{lemma}\label{lem:Witt_properties}
	The following assertions hold for both $\mathbb{W}_n$ and $\widehat{\mathbb{W}}_n$:
	\begin{enumerate}
		\item	$\star$ is commutative, associative, and unital with unit $U$.
		\item	$F_1=V_1=\id$, $F_sF_r=F_{sr}$, $V_sV_r=V_{sr}$.
		\item	$F_sV_s=s\cdot \id$. If $(s,r)=1$, then $F_sV_r=V_rF_s$.
		\item	$V_s\circ \star\circ (\id\otimes F_s) = \star\circ (V_s\otimes\id)$.
		\item	$\star\circ (F_s\otimes F_s)=F_s\circ \star$.
	\end{enumerate}
\end{lemma}

\begin{proof}
	The statement for $\mathbb{W}_n$ is well-known (see e.g. \cite{Hes15}), and the statement for $\widehat{\mathbb{W}}_n$ follows easily from this.
\end{proof}

\begin{lemma}
	For any $n\geq 0$, $s\geq 1$, and $0\leq m\leq sn$, the morphism $V_sF_s\colon \mathbb{W}_{sn}\to \mathbb{W}_{sn}$ descends to $\mathbb{W}_{m}\to \mathbb{W}_{m}$.
	The same also holds for $\widehat{\mathbb{W}}_n$.
\end{lemma}

\begin{proof}
	Since $s(\lfloor m/s\rfloor+1)\geq m+1$, the morphism
	$V_s\colon \mathbb{W}_{\lfloor m/s\rfloor}\to \mathbb{W}_{s\lfloor m/s\rfloor}$
	lifts to
	$\mathbb{W}_{\lfloor m/s\rfloor}\to \mathbb{W}_m$.
	Composing with $\mathbb{W}_m\twoheadrightarrow \mathbb{W}_{s\lfloor m/s\rfloor}\xrightarrow{F_s} \mathbb{W}_{\lfloor m/s\rfloor}$, we get the desired morphism.
\end{proof}

Suppose that $\ch(k)=p>0$.
For a prime number $\ell$ different from $p$, we have an endomorphism
$\ell^{-1}V_\ell F_\ell$ of $(\widehat{\mathbb{W}}_n)_{(p)}=\widehat{\mathbb{W}}_n\otimes \mathbb{Z}_{(p)}$.
By Lemma \ref{lem:Witt_properties}, $\ell^{-1}V_\ell F_\ell$ is idempotent and hence defines a direct summand $\Im(\ell^{-1}V_\ell F_\ell)$ of $(\widehat{\mathbb{W}}_n)_{(p)}$.
For $n\geq 1$, the \emph{presheaf of $p$-typical Witt vectors of length $n$} is defined by
$$
	W_n:=(\widehat{\mathbb{W}}_{p^{n-1}})_{(p)}/\textstyle\sum_{\ell\neq p}\Im (\ell^{-1}V_\ell F_\ell) \in \PSh(\Cor_k)
$$
where the sum is taken over all prine numbers different from $p$.

\begin{lemma}
	The above definition of $W_n$ coincides with the usual one (e.g. \cite{Hes15}).
\end{lemma}

\begin{proof}
	Since $\ell^{-1} V_\ell F_\ell(0,a,m)=(0,a,m)$, the subpresheaf $(0\oplus \mathbb{G}_m\oplus \mathbb{Z})_{(p)}$ of $(\widehat{\mathbb{W}}_{p^{n-1}})_{(p)}$ is contained in $\Im (\ell^{-1} V_\ell F_\ell)$.
	Moreover, $\mathbb{W}_{p^{n-1}}$ is a presheaf of $\mathbb{Z}_{(p)}$-modules since $\ch(k)=p$.
	Therefore we have
	$$
		W_n=\mathbb{W}_{p^{n-1}}/\textstyle\sum_{\ell\neq p}\Im (\ell^{-1} V_\ell F_\ell).
	$$
	Now the claim follows from the $p$-typical decomposition of $\mathbb{W}_{p^{n-1}}$ (see \cite[Proposition 1.10]{Hes15}).
\end{proof}

By Lemma \ref{lem:Witt_properties}, the homomorphisms $U$, $\star$, $F_p$, and $V_p$ descend to
\begin{align*}
	U&\colon \mathbb{Z}_{(p)}\to W_n,&\star&\colon W_n\otimes W_n\to W_n,\\
	F&\colon W_{n+1}\to W_n,&V&\colon W_n\to W_{n+1}
\end{align*}
and we have the following
\begin{lemma}\label{lem:p_Witt_properties}
	The following assertions hold for $W_n$:
	\begin{enumerate}
		\item	$\star$ is commutative, associative, and unital with unit $U$.
		\item	$FV=p\cdot \id$.
		\item	$V\circ\star\circ(\id\otimes F) = \star\circ(V\otimes\id)$.
		\item	$\star\circ (F\otimes F)=F\circ \star$.
	\end{enumerate}
\end{lemma}

\subsection{Motivic construction}

\begin{definition}
For $n\geq 0$, we define $\mathbb{W}^+_n, \widehat{\mathbb{W}}^+_n\in \PSh(\MCor^\mathbb{Q}_k)$ by
$$
	\mathbb{W}^+_n:= \varinjlim_{\varepsilon>0}\mathbb{Z}_\tr(\mathbb{P}^1,(n+\varepsilon)[\infty])/\mathbb{Z},\quad
	\widehat{\mathbb{W}}^+_n:= \varinjlim_{\varepsilon>0} \mathbb{Z}_\tr(\mathbb{P}^1,\varepsilon[0]+(n+\varepsilon)[\infty])
$$
where $\mathbb{Z}$ is viewed as a direct summand of $\mathbb{Z}_\tr(\mathbb{P}^1,(n+\varepsilon)[\infty])$ via $i_0\colon \Spec k\to \mathbb{A}^1$.
\end{definition}

\begin{definition}\label{def:motivic_operations}
Let $n\geq 0$.
We define the unit, the multiplication, the Frobenius, and the Verschiebung on $\mathbb{W}^+_n$ as follows.
\begin{enumerate}
\item We define $U\colon \mathbb{Z}\to \mathbb{W}^+_n$ to be the morphism induced by $i_1\colon \Spec k\to \mathbb{A}^1$.
\item We define $\star\colon \mathbb{W}^+_n \otimes \mathbb{W}^+_n \to \mathbb{W}^+_n$ to be the morphism induced by the multiplication map $\mu \colon \AA^1 \times \AA^1 \to \AA^1;\; (x,y) \mapsto xy$ (see Lemma \ref{cube_multiplication_admissible}).
\item For $s\geq 1$, we define $F_s\colon \mathbb{W}^+_{sn}\to \mathbb{W}^+_n$ to be the morphism induced by $\rho_s\colon \AA^1 \to \AA^1;\;x \mapsto x^s$.
\item For $s\geq 1$, we define $V_s\colon \mathbb{W}^+_{n}\to \mathbb{W}^+_{sn}$ to be the morphism induced by ${}^t\rho_s$.
\end{enumerate}
Similarly, we define the unit, the multiplication, the Frobenius, and the Verschiebung on $\widehat{\mathbb{W}}^+_n$.
For $X=\Spec A\in \Sm_k^\aff$ and $a\in A$ (resp. $a\in A^\times$), we write $[a]$ for the element of $h_0\mathbb{W}^+_n(X)$ (resp. $h_0\widehat{\mathbb{W}}^+_n(X)$) represented by the morphism $a\colon X\to \mathbb{A}^1$ (resp. $a\colon X\to \mathbb{A}^1\setminus\{0\}$).
\end{definition}

\begin{lemma}\label{lem:motivic_Witt_properties}
	The following assertions hold for both $\mathbb{W}^+_n$ and $\widehat{\mathbb{W}}^+_n$:
	\begin{enumerate}
		\item	$\star$ is commutative, associative, and unital with unit $U$.
		\item	$F_1=V_1=\id$, $F_sF_r=F_{sr}$, $V_sV_r=V_{sr}$.
		\item	$F_sV_s=s\cdot \id$. If $(s,r)=1$, then $F_sV_r=V_rF_s$.
		\item	$V_s\circ\star\circ (\id\otimes F_s) = \star\circ (V_s\otimes \id)$.
		\item	$\star\circ(F_s\otimes F_s)=F_s\circ \star$.
	\end{enumerate}
\end{lemma}

\begin{proof}
	(1) follows from the corresponding properties of $\mu\colon \mathbb{A}^1\times\mathbb{A}^1\to \mathbb{A}^1$.
	(2) follows from $\rho_1=\id$ and $\rho_r\circ\rho_s=\rho_{rs}$.
	The first assertion in (3) follows from $\rho_s{}^t\circ\rho_s=s\cdot\id$.
	The second assertion follows from the fact that $\rho_s{}^t\circ\rho_r$ and ${}^t\rho_r\circ\rho_s$ are both represented by the cycle on $\mathbb{A}^2$ defined by $x^s=y^r$.
	To prove (4), it suffices to show that the following diagram in $\Cor_k$ is commutative:
	$$
		\xymatrix{
			\mathbb{A}^1\times\mathbb{A}^1\ar[rr]^-{{}^t\rho_s\times\id}\ar[d]^-{\id\times \rho_s}		&&\mathbb{A}^1\times\mathbb{A}^1\ar[d]^-\mu\\
			\mathbb{A}^1\times\mathbb{A}^1\ar[r]^-\mu		&\mathbb{A}^1\ar[r]^-{{}^t\rho_s}		&\mathbb{A}^1.
		}
	$$
	One can easily check that both compositions are represented by the cycle on $\mathbb{A}^3$ defined by $xy^s=z^s$.
	Finally, (5) follows from $\rho_s\circ \mu=\mu\circ(\rho_s\times\rho_s)$.
\end{proof}

\begin{lemma}
	For any $n\geq 0$, $s\geq 1$ and $0\leq m\leq sn$, the morphism $V_sF_s\colon \mathbb{W}^+_{sn}\to \mathbb{W}^+_{sn}$ descends to $\mathbb{W}^+_m\to \mathbb{W}^+_m$.
	The same also holds for $\widehat{\mathbb{W}}_n$.
\end{lemma}

\begin{proof}
	To prove the first statement, it suffices to show that the finite correspondence ${}^t\rho_s\circ\rho_s\in \Cor_k(\mathbb{A}^1,\mathbb{A}^1)$ is contained in $\MCor^\mathbb{Q}_k((\mathbb{P}^1,r[\infty]),(\mathbb{P}^1,r[\infty]))$ for every $r\in \mathbb{Q}_{>0}$.
	This follows from the fact that ${}^t\rho_s\circ\rho_s$ is represented by the cycle on $\mathbb{A}^2$ defined by $x^s=y^s$, which is symmetric in $x$ and $y$.
	The second statement can be proved similarly.
\end{proof}

Suppose that $\ch(k)=p>0$.
For a prime number $\ell$ different from $p$, we have an endomorphism
$\ell^{-1}V_\ell F_\ell$ of $(\widehat{\mathbb{W}}^+_{n})_{(p)}$.
By Lemma \ref{lem:motivic_Witt_properties}, $\ell^{-1}V_\ell F_\ell$ is an idempotent and hence defines a direct summand $\Im(\ell^{-1}V_\ell F_\ell)$ of $(\widehat{\mathbb{W}}^+_{n})_{(p)}$.
For $n\geq 1$, we define
\begin{equation}\label{def:Wn+}
	W^+_n:=(\widehat{\mathbb{W}}^+_{p^{n-1}})_{(p)}/\textstyle\sum_{\ell\neq p}\Im (\ell^{-1} V_\ell F_\ell) \in \PSh(\MCor^\mathbb{Q}_k).
\end{equation}
For $X=\Spec A\in \Sm_k^\aff$ and $a\in A^\times$, we write $[a]$ for the image of $[a]\in h_0\widehat{\mathbb{W}}^+_{p^{n-1}}(X)$ in $h_0W^+_n(X)$.
By Lemma \ref{lem:motivic_Witt_properties}, the morphisms $U$, $\star$, $F_p$, and $V_p$ descend to
\begin{align*}
	U&\colon \mathbb{Z}_{(p)}\to W^+_n,&\star&\colon W^+_n\otimes W^+_n\to W^+_n,\\
	F&\colon W^+_{n+1}\to W^+_n,&V&\colon W^+_n\to W^+_{n+1}
\end{align*}
and we have the following

\begin{lemma}\label{lem:motivic_p_Witt_properties}
	The following assertions hold for $W^+_n$:
	\begin{enumerate}
		\item	$\star$ is commutative, associative, and unital with unit $U$.
		\item	$FV=p\cdot \id$.
		\item	$V\circ(\id\otimes F) = \star\circ (V\otimes \id)$.
		\item	$\star\circ(F\otimes F)=F\circ \star$.
	\end{enumerate}
\end{lemma}

\subsection{Comparison}

\begin{theorem}\label{thm:ab-comparison}
	Let $n\geq 0$.
	\begin{enumerate}
		\item	There is an isomorphism
				$\varphi\colon h_0\mathbb{W}^+_n\xrightarrow{\sim} \mathbb{W}_n$
				in $\PSh(\Cor_k^\aff)$.
				For $X=\Spec A\in \Sm_k^\aff$, the image of $\div(x^m+a_1x^{m-1}+\dots+a_m) \in h_0\mathbb{W}^+_n(X)$ is $1+a_1t+\dots+a_mt^m \in \bW_{n}(X)$.
				In particular, we have $\varphi_X([a])=[a]$ for $a\in A$.
 		\item	There is an isomorphism
				$\widehat{\varphi}\colon h_0\widehat{\mathbb{W}}^+_n\xrightarrow{\sim} \widehat{\mathbb{W}}_n$
				in $\PSh(\Cor_k^\aff)$.
				For $X=\Spec A\in \Sm_k^\aff$, the image of $\div(x^m+a_1x^{m-1}+\dots+a_m)\in h_0\widehat{\mathbb{W}}^+_n(X)$ is $(1+a_1t+\dots+a_mt^m, (-1)^ma_m, m) \in \widehat{\bW}_{n} (X)$.
				In particular, we have $\widehat{\varphi}_X([a])=[a]$ for $a\in A^\times$.
	\end{enumerate}
\end{theorem}

\begin{proof}
	We prove only (1); the proof for (2) is similar.
	First we construct an isomorphism
	$$
		\varphi_X\colon h_0\mathbb{W}^+_n(X)\xrightarrow{\sim} \mathbb{W}_n(X)
	$$
	for $X=\Spec A\in \Sm_k^\aff$.
	We may assume that $X$ is connected.
	For $\varepsilon \in (0,1)\cap \mathbb{Q}$ we have
	\begin{align*}h_0(\mathbb{P}^1,(n+\varepsilon)[\infty])(X)
	&\simeq h_0(\mathbb{P}^1,(n+1)[\infty])(X)	&\text{(Theorem \ref{geometric Hasse-Arf for curves})}\\
	&\simeq \Pic(\mathbb{P}^1_X,(n+1)[\infty]).&\text{(Theorem \ref{modulus_curve_suslin_homology})}
	\end{align*}
	The last term can be easily identified with $\mathbb{W}_n(X)\oplus \mathbb{Z}$ (see e.g. \cite[Proposition 1.1]{Koi2}).
	Therefore we have $h_0\mathbb{W}^+_n(X) \simeq (\mathbb{W}_n(X)\oplus \mathbb{Z})/\mathbb{Z}\simeq \mathbb{W}_n(X)$.

	It remains to show that this isomorphism is compatible with transfers.
	Let $X,Y\in \Sm_k^\aff$ and $\alpha\in \Cor_k(X,Y)$.
	It suffices to prove that the following diagram is commutative:
	\begin{align}\label{diag_1}
		\xymatrix{
			h_0\mathbb{W}^+_n(Y)\ar[r]^-{\varphi_{Y}} \ar[d]^-{\alpha^*}	&\mathbb{W}_n(Y)\ar[d]^-{\alpha^*}\\
			h_0\mathbb{W}^+_n(X)\ar[r]^-{\varphi_{X}}						&\mathbb{W}_n(X).
		}
	\end{align}
	We may assume that $X,Y$ are connected.
	Let $K$ be an algebraic closure of $k(X)$, and consider the same problem for $X_K,Y_K\in \Sm_k^\aff$ and $\alpha_K\in \Cor_K(X_K,Y_K)$:
	\begin{align}\label{diag_2}
		\xymatrix{
			h_0\mathbb{W}^+_n(Y_K)\ar[r]^-{\varphi_{Y_K}} \ar[d]^-{\alpha_K^*}	&\mathbb{W}_n(Y_K)\ar[d]^-{\alpha_K^*}\\
			h_0\mathbb{W}^+_n(X_K)\ar[r]^-{\varphi_{X_K}}						&\mathbb{W}_n(X_K).
		}
	\end{align}
	Then there is a natural morphism of diagrams from (\ref{diag_1}) to (\ref{diag_2}).
	Since the composite
	$$
		\mathbb{W}_n(X) \to \mathbb{W}(X_K)\xrightarrow{\Delta_K^*} \mathbb{W}_n(K)
	$$
	is injective, we may assume that $k$ is algebraically closed and $X=\Spec k$.
	In this case, $\alpha$ can be written as a $\mathbb{Z}$-linear combination of morphisms $\Spec k\to Y$, so the assertion is obvious.
\end{proof}

\begin{proposition}\label{prop:ring-comparison}
	The isomorphisms $\varphi\colon h_0\mathbb{W}^+_n\xrightarrow{\sim} \mathbb{W}_n$ and $\widehat{\varphi} \colon h_0\widehat{\mathbb{W}}^+_n\xrightarrow{\sim} \widehat{\mathbb{W}}_n$ in $\PSh(\Cor_k^\aff)$ from Theorem \ref{thm:ab-comparison} are compatible with the unit, the multiplication, the Frobenius, and the Verschiebung.
\end{proposition}

\begin{proof}
	We prove only the statement for $\varphi$; the proof for $\widehat{\varphi}$ is similar.
	It suffices to prove that for each $X\in \Sm_k^\aff$, the isomorphism $\varphi_X\colon h_0\mathbb{W}^+_n(X)\xrightarrow{\sim} \mathbb{W}_n(X)$ is compatible with these operations.
	As in the proof of Theorem \ref{thm:ab-comparison}, we may assume that $k$ is algebraically closed and $X=\Spec k$.
	In this case, $\mathbb{W}_n(k)$ is generated by elements of the form $[a]$ with $a\in k^\times$.
	For $a,b\in k^\times$, we have 
	$$
		U(1)=[1],\quad [a]\star[b]=[ab],\quad F_s([a])=[a^s],\quad V_s([a])=1-at^s
	$$ in $\mathbb{W}_n(k)$ and
	$$
		U(1)=[1],\quad [a]\star[b]=[ab],\quad F_s([a])=[a^s],\quad V_s ([a]) = \div(x^s-a)
	$$
	in $h_0\mathbb{W}^+_n(k)$.
	This proves the claim.
\end{proof}

\begin{corollary}\label{cor:p_typical_comparison}
	Suppose that $\ch(k)=p>0$.
	For $n\geq 1$, there is an isomorphism
	$$
		\varphi^{(p)}\colon h_0 W^+_n\xrightarrow{\sim} W_n
	$$
	in $\PSh(\Cor_k^\aff)$ which is compatible with the unit, the multiplication, the Frobenius, and the Verschiebung.
	If $X=\Spec A\in \Sm_k^\aff$, then we have $\varphi^{(p)}_X([a])=[a]$ for $a\in A^\times$.
\end{corollary}

\subsection{Application to torsion and divisibility of reciprocity sheaves}

In this subsection, we give an application of the motivic presentation of the ring of big Witt vectors to reciprocity sheaves.

Let $F\in \RSC_k$.
Imitating the construction in \cite{Miyazaki-19}, we define $N_nF,NF\in \PSh(\Cor_k)$ by
\begin{align*}
	N_nF(X)&:=\varprojlim_{\varepsilon>0}\Ker(h^0F((X,\emptyset)\otimes(\mathbb{P}^1,(n+\varepsilon)[\infty]))\xrightarrow{i_0^*}F(X)),\\
	NF(X)&:=\Ker(F(X\times\mathbb{A}^1)\xrightarrow{i_0^*}F(X)).
\end{align*}
Here, we regard $(X,\emptyset)$ as a pro-object in $\MCor_k^\mathbb{Q}$ via the compactification functor.
In other words, we define $N_nF:=h_0\intHom(\mathbb{W}^+_n,h^0F)$ and $NF:=\varinjlim_{n\geq 0}N_nF$.
The presheaf $NF$ measures how far $F$ is from being $\mathbb{A}^1$-invariant.
The composition
$$
	\mathbb{W}^+_n\otimes \mathbb{W}^+_n\otimes \intHom(\mathbb{W}^+_n,h^0F)\xrightarrow{\star\otimes \id}
	\mathbb{W}^+_n\otimes \intHom(\mathbb{W}^+_n,h^0F)\xrightarrow{\mathrm{ev}}
	h^0F
$$
induces an action
$$
	\mathbb{W}^+_n\otimes \intHom(\mathbb{W}^+_n,h^0F) \to \intHom(\mathbb{W}^+_n,h^0F)
$$
by adjunction.
Taking $h_0$, we obtain an action
$h_0\mathbb{W}^+_n\otimes N_nF\to N_nF$.
Since we have an isomorphism of rings $h_0\mathbb{W}^+_n(X)\simeq \mathbb{W}_n(X)$ for $X\in \Sm_k^\aff$ by Theorem \ref{thm:ab-comparison} and Proposition \ref{prop:ring-comparison}, we obtain the following

\begin{theorem}
	Let $F\in \RSC_k$ and $X\in \Sm_k^\aff$.
	Then $N_nF(X)$ has a canonical structure of a $\mathbb{W}_n(X)$-module, which is natural in $X$ and $n$.
	In particular, $NF(X)$ has a canonical structure of a $\mathbb{W}(X)$-module, where $\mathbb{W}(X)=\varprojlim_{n\geq 0}\mathbb{W}_n(X)$.
\end{theorem}

\begin{corollary}
	Let $F\in \RSC_k$, $X\in \Sm_k^\aff$, and let $p$ be a prime number.
	\begin{enumerate}
		\item	If $p$ is invertible in $k$, then $N_nF(X)$ and $NF(X)$ are uniquely $p$-divisible.
		\item	If $\ch(k)=p>0$, then $N_nF(X)$ and $NF(X)$ are $p$-groups.
	\end{enumerate}
\end{corollary}

\begin{proof}
	This follows from the fact that if $p$ is invertible (resp. nilpotent) in a ring $A$, then $p$ is also invertible (resp. nilpotent) in $\mathbb{W}_n(A)$; see \cite[Lemma 1.9 and Proposition 1.10]{Hes15}.
\end{proof}

\begin{corollary}[\protect{\cite[Theorem 1.3]{BCKS}}]\label{BCKSgen}
	Let $F\in \RSC_k$ and assume that $F$ is separated for the Zariski topology.
	\begin{enumerate}
		\item	If $\ch(k)=0$ and $F\otimes\mathbb{Q}=0$, then $F$ is $\mathbb{A}^1$-invariant.
		\item	If $\ch(k)=p>0$ and $F$ is $p$-torsion-free, then $F$ is $\mathbb{A}^1$-invariant.
	\end{enumerate}
\end{corollary}

\section{Motivic construction of the de Rham-Witt complex}
\label{sec:dRWitt}

Throughout this section, we assume that $k$ is \emph{perfect} and $\ch(k)=p\geq 3$
\footnote{
Here, the perfectness of $k$ is assumed in order to use the transfer structure on the de Rham-Witt complex.
The assumption on $\ch(k)$ will be used in the proofs of Lemma \ref{FdV} and Theorem \ref{thm:theta}.
}.
In this section, we construct the de Rham-Witt complex of smooth $k$-schemes using $\mathbb{Q}$-modulus pairs.

\subsection{De Rham-Witt complex}

First we recall the definition of the de Rham-Witt complex.
Here, we follow the axiomatization due to \cite{Hesselholt-Madsen}.

\begin{definition}
	Let $A$ be a $\mathbb{Z}_{(p)}$-algebra.
	A \emph{Witt complex} over $A$ is a tuple $(E_\bullet^*,F,V, \lambda)$ where
	\begin{enumerate}
		\item	$\cdots \to E_n^*\xrightarrow{R} E_{n-1}^*\to \cdots\to E_0^*$ is a sequence of CDGAs,
		\item	$F\colon E_{n+1}^*\to E_n^*$ is a graded ring homomorphism compatible with $R$,
		\item	$V\colon E_n^*\to E_{n+1}^*$ is a graded group homomorphism compatible with $R$,
		\item	$\lambda\colon W_n(A)\to E_n^0$ is a ring homomorphism compatible with $R$, $F$, and $V$,
	\end{enumerate}
	such that the following relations hold:
	$$
		V(x\star F(y))=V(x)\star y,\quad FdV=d,\quad FV=p\cdot\id, \quad F(d\lambda[a])=\lambda[a^{p-1}]\star d\lambda[a]\;(a\in A).
	$$
\end{definition}

The category of Witt complexes over $A$ is known to have an initial object $W_n\Omega^*(A)$ and it is called the \emph{de Rham-Witt complex} of $A$.
We have $W_n\Omega^0(A)\simeq W_n(A)$.
It follows from the construction that $W_n\Omega^q(A)$ is a quotient of $\Omega^q_{W_n(A)}$.

The presheaf $\Spec A\mapsto W_n\Omega^q(A)$ on $\Sm_k^\aff$ extends to an \'etale sheaf $W_n\Omega^q$ on $\Sm_k$ having {\it global injectivity}
\footnote{
This means that for any $X\in \Sm_k$ and a dense open subset $U\subset X$, the map $W_n\Omega^q(X)\to W_n\Omega^q(U)$ is injective.
This follows from \cite[I, Corollaire 3.9]{Illusie_dRW}.
}.
Moreover, since $k$ is assumed to be perfect, $W_n\Omega^q$ can be regarded as an object of $\PSh(\Cor_k)$ \cite[Theorem B.2.1]{KSY1}.
It is compatible with the trace maps defined in \cite[Theorem 2.6]{Rulling-thesis} (see \cite[Section 7.9]{RS21}).

\subsection{Motivic construction}

\begin{definition}\label{def:Gm+}
	We define $\mathbb{G}_m^+\in \PSh(\MCor^\mathbb{Q}_k)$ by
	$$
		\mathbb{G}_m^+:=\varinjlim_{\varepsilon>0}\mathbb{Z}_\tr(\mathbb{P}^1, \varepsilon[0]+\varepsilon[\infty])/\mathbb{Z}
	$$
	where $\mathbb{Z}$ is viewed as a direct summand of $\mathbb{Z}_\tr(\mathbb{P}^1,\varepsilon[0]+\varepsilon[\infty])$ via $i_1\colon \Spec k\to \mathbb{A}^1\setminus\{0\}$.
\end{definition}

\begin{lemma}\label{Gm_anti_commutative}
	The following assertions hold.
	\begin{enumerate}
	\item	Let $F_s\colon \mathbb{G}_m^+\to \mathbb{G}_m^+$ be the morphism induced by $\rho_s$.
			Then we have $h_0(F_s)=s\cdot \id$.
	\item	Let $V_s\colon \mathbb{G}_m^+\to \mathbb{G}_m^+$ be the morphism induced by ${}^t\rho_s$.
			If $s$ is odd, then $h_0(V_s)=\id$.
	\item	For $X=\Spec A \in \Sm_k^\aff$ and $a\in A^\times$, we write $[a]$ for the element of $h_0(\mathbb{G}_m^+)(X)$ represented by the morphism $a\colon X\to \mathbb{A}^1\setminus\{0\}$.
			 Then we have $[a]+[b]=[ab]$.
	\item	Let $\tau\colon \mathbb{G}_m^{+\otimes 2}\to \mathbb{G}_m^{+\otimes 2}$ be the morphism induced by
			$
			(\mathbb{A}^1\setminus\{0\})^2\to (\mathbb{A}^1\setminus\{0\})^2;
				\; (x,y)\mapsto (y,x).
			$
			Then we have $h_0(\tau)=-\id$.
	\item	Let $\delta\colon \mathbb{G}_m^+\to \mathbb{G}_m^{+\otimes 2}$ be the morphism induced by
			$
			\mathbb{A}^1\setminus\{0\}\to (\mathbb{A}^1\setminus\{0\})^2;
				\; x\mapsto (x,x).
			$
			Then we have $2\cdot h_0(\delta)=0$.
	\end{enumerate}
\end{lemma}

\begin{proof}
	\noindent
	(1) Define $\Gamma\subset (\mathbb{A}^1\setminus\{0\})\times\mathbb{A}^1\times(\mathbb{A}^1\setminus\{0\})=\Spec k[x^\pm,t,y^\pm]$ by the equation
	$$
		(1-t)(y-x^s)(y-1)^{s-1}+t(y-x)^s=0.
	$$
	Regarding the left hand side as a polynomial in $y$, the leading term and the constant term have invertible coefficients.
	Therefore $\Gamma$ is finite locally free over $(\mathbb{A}^1\setminus\{0\})\times\mathbb{A}^1$ and hence defines a finite correspondence $[\Gamma]\in \Cor_k((\mathbb{A}^1\setminus\{0\})\times\mathbb{A}^1, \mathbb{A}^1\setminus\{0\})$.
	The closure $\overline{\Gamma}\subset \mathbb{P}^1\times\mathbb{P}^1\times\mathbb{P}^1$ of $\Gamma$ satisfies the following condition for $\varepsilon\in \mathbb{Q}_{>0}$:
	$$
		\pr_1^*(s\varepsilon[0]+s\varepsilon[\infty])|_{\overline{\Gamma}^N}+\pr_2^*[\infty]|_{\overline{\Gamma}^N}\geq \pr_3^*(\varepsilon[0]+\varepsilon[\infty])|_{\overline{\Gamma}^N}.
	$$
	Therefore $\Gamma$ induces a morphism $\mathbb{G}_m^+\otimes\mathbb{Z}_\tr(\bcube)\to \mathbb{G}_m^+$ whose restriction to $t=0$ (resp. $t=1$) is $F_s$ (resp. $s\cdot \id$).
	This shows that $h_0(F_s)=s\cdot \id$.
	
	\noindent
	(2) Define $\Gamma\subset (\mathbb{A}^1\setminus\{0\})\times\mathbb{A}^1\times(\mathbb{A}^1\setminus\{0\})=\Spec k[x^\pm,t,y^\pm]$ by the equation
	$$
		(1-t)(y^s-x)+t(y-x)(y-1)^{s-1}=0.
	$$
	Regarding the left hand side as a polynomial in $y$, the leading term and the constant term have invertible coefficients \emph{if $s$ is odd}.
	Therefore $\Gamma$ is finite locally free over $(\mathbb{A}^1\setminus\{0\})\times \mathbb{A}^1$ and hence defines a finite correspondence $[\Gamma]\in \Cor_k((\mathbb{A}^1\setminus\{0\})\times\mathbb{A}^1, \mathbb{A}^1\setminus\{0\})$.
	As in the proof of (1), induces a morphism $\mathbb{G}_m^+\otimes\mathbb{Z}_\tr(\bcube)\to \mathbb{G}_m^+$ whose restriction to $t=0$ (resp. $t=1$) is $V_s$ (resp. $\id$).
	This shows that $h_0(V_s)=\id$.
	
	\noindent
	(3) Define $\Gamma\subset X\times\mathbb{A}^1\times(\mathbb{A}^1\setminus\{0\})=\Spec A[t,x^\pm]$ by the equation
 	$$
		(1-t)(x-a)(x-b)+t(x-ab)(x-1)=0.
	$$
	Regarding the left hand side as a polynomial in $x$, the leading term and the constant term have invertible coefficients.
	Therefore $\Gamma$ is finite locally free over $\Spec A[t]$ and hence defines a finite correspondence $[\Gamma]\in \Cor_k(X\times\mathbb{A}^1, \mathbb{A}^1\setminus\{0\})$.
	This induces a morphism $\mathbb{Z}_\tr(X)\otimes \mathbb{Z}_\tr(\bcube)\to \mathbb{G}_m^+$ whose restriction to $t=0$ (resp. $t=1$) is $[a]+[b]$ (resp. $[ab]$).
	This shows that $[a]+[b]=[ab]$.
	
	\noindent
	(4) Define $\Gamma\subset (\mathbb{A}^1\setminus\{0\})^2\times\mathbb{A}^1\times(\mathbb{A}^1\setminus\{0\})^2=\Spec k[x^\pm,y^\pm,t,z^\pm,w^\pm]$ by the equations
	$$
		z+w=(1-t)(x+y)+t(xy+1),\quad zw=xy.
	$$
	One can check that $\Gamma$ is finite locally free over $(\mathbb{A}^1\setminus\{0\})^2\times\mathbb{A}^1$ and hence defines a finite correspondence $[\Gamma]\in \Cor_k((\mathbb{A}^1\setminus\{0\})^2\times\mathbb{A}^1, (\mathbb{A}^1\setminus\{0\})^2)$.
	This induces a morphism $\mathbb{G}_m^{+\otimes 2}\otimes \mathbb{Z}_\tr(\bcube)\to \mathbb{G}_m^{+\otimes 2}$ whose restriction to $t=0$ (resp. $t=1$) is $\id+\tau$ (resp. $0$).
	This shows that $h_0(\id+\tau)=0$ and hence $h_0(\tau)=-\id$.
	
	\noindent
	(5) We have $\tau\circ \delta=\delta$.
	Since $h_0(\tau)=-\id$ by (4), we get $-h_0(\delta)=h_0(\delta)$ and hence $2\cdot h_0(\delta)=0$.
\end{proof}

\begin{remark}
	Actually, we can prove that $h_0(\mathbb{G}_m^{+\otimes q})\cong K_q^M$ holds in $\PSh(\Cor_k)$, where $K_q^M$ is the unramified sheaf of Milnor $K$-groups; it is essentially a corollary of \cite[Theorem 3.4]{SV00}.
\end{remark}

\begin{definition}
Let $n\geq 1$.
\begin{enumerate}
	\item
	For $q,r\geq 0$, we define the \emph{multiplication} on $W_n^+\otimes \mathbb{G}_m^{+\otimes *}$ by
	\begin{align*}
	\star\colon (W^+_n\otimes\mathbb{G}_m^{+\otimes q})
	\otimes(W^+_n\otimes\mathbb{G}_m^{+\otimes r})
	&\xrightarrow{\sim}
	W^+_n\otimes W^+_n
	\otimes\mathbb{G}_m^{+\otimes q}\otimes\mathbb{G}_m^{+\otimes r}\\
	&\xrightarrow{\star\otimes \id}
	W^+_n\otimes\mathbb{G}_m^{+\otimes (q+r)}.
	\end{align*}
	For each $X\in \Sm_k^\aff$, this makes $h_0 (W_n^+\otimes \mathbb{G}_m^{+\otimes *})(X)$ into a graded ring.
	By Lemma \ref{Gm_anti_commutative} (4), this multiplication is graded commutative.
	Moreover, it is a $\mathbb{Z}_{(p)}$-algebra since $h_0 W_n^+(X)\simeq W_n(X)$ is so.
	\item
	For $q\geq 0$, we define the \emph{Frobenius} by
	$F:=F\otimes \id\colon W^+_{n+1}\otimes\mathbb{G}_m^{+\otimes *}
	\to W^+_n\otimes\mathbb{G}_m^{+\otimes *}$
	and the \emph{Verschiebung} by
	$V:=V\otimes \id\colon W^+_n\otimes\mathbb{G}_m^{+\otimes *}
	\to W^+_{n+1}\otimes\mathbb{G}_m^{+\otimes *}$.
	By Lemma \ref{lem:motivic_p_Witt_properties}, $F$ is a graded ring homomorphism, and we have $V(x\star F(y))=V(x)\star y$, $FV=p$.
	\item
	We define $\lambda\colon W_n\xrightarrow{\sim} h_0 W^+_n$ in $\PSh(\Cor_k^\aff)$ to be the inverse of the isomorphism $\varphi^{(p)}$ given in Corollary \ref{cor:p_typical_comparison}.
	This is a ring homomorphism compatible with $F$ and $V$.
\end{enumerate}
\end{definition}

\begin{definition}
	Let $\Gamma$ be the graph of the diagonal morphism $\Delta\colon \mathbb{A}^1\setminus\{0\}\to (\mathbb{A}^1\setminus\{0\})^2$.
	Then the closure $\overline{\Gamma}\subset \mathbb{P}^1\times \mathbb{P}^1\times \mathbb{P}^1$ of $\Gamma$ satisfies the following condition for $\varepsilon\in \mathbb{Q}_{>0}$:
	$$
		\pr_1^*(2\varepsilon[0]+(n+2\varepsilon)[\infty])|_{\overline{\Gamma}^N}\geq \pr_2^*(\varepsilon[0]+(n+\varepsilon)[\infty])|_{\overline{\Gamma}^N}+\pr_3^*(\varepsilon[0]+\varepsilon[\infty])|_{\overline{\Gamma}^N}.
	$$
	We define $d\colon \widehat{\mathbb{W}}^+_n\to \widehat{\mathbb{W}}^+_n\otimes \mathbb{G}_m^+$ to be the morphism induced by $\Gamma$.
\end{definition}

\begin{lemma}\label{delta_and_Verschiebung}
	For any prime number $\ell$ different from $p$, the following diagram becomes commutative after applying $h_0$:
	$$
		\xymatrix{
			\widehat{\mathbb{W}}^+_{p^{n-1}}\ar[r]^-{d}\ar[d]^-{\ell V_\ell F_\ell}	&
			\widehat{\mathbb{W}}^+_{p^{n-1}}\otimes\mathbb{G}_m^+\ar[d]^-{\ell V_\ell F_\ell\otimes \id}\\
			\widehat{\mathbb{W}}^+_{p^{n-1}}\ar[r]^-{d}&
			\widehat{\mathbb{W}}^+_{p^{n-1}}\otimes\mathbb{G}_m^+.
		}
	$$
\end{lemma}

\begin{proof}
	Recall from Remark \ref{h0_lax_monoidal} that we have a canonical epimorphism $h_0(\widehat{\mathbb{W}}^+_{p^{n-1}})\otimes h_0(\mathbb{G}_m^+)\twoheadrightarrow h_0(\widehat{\mathbb{W}}^+_{p^{n-1}}\otimes\mathbb{G}_m^+)$.
	By Lemma \ref{Gm_anti_commutative} (1), we have $h_0(\id\otimes F_\ell)=\ell\cdot \id$ on $h_0(\widehat{\mathbb{W}}^+_{p^{n-1}}\otimes\mathbb{G}_m^+)$.
	Therefore it suffices to show that the following diagram is commutative:
	$$
		\xymatrix{
			\widehat{\mathbb{W}}^+_{p^{n-1}}\ar[r]^-d\ar[d]^-{V_\ell F_\ell}&
			\widehat{\mathbb{W}}^+_{p^{n-1}}\otimes\mathbb{G}_m^+\ar[r]^-{\id\otimes F_\ell}&
			\widehat{\mathbb{W}}^+_{p^{n-1}}\otimes\mathbb{G}_m^+\ar[d]^-{V_\ell F_\ell\otimes \id}\\
			\widehat{\mathbb{W}}^+_{p^{n-1}}\ar[r]^-d&
			\widehat{\mathbb{W}}^+_{p^{n-1}}\otimes\mathbb{G}_m^+\ar[r]^-{\id\otimes F_\ell}&
			\widehat{\mathbb{W}}^+_{p^{n-1}}\otimes\mathbb{G}_m^+.
		}
	$$
	This is further reduced to the commutativity of the following diagram in $\Cor_k$:
	$$
		\xymatrix{
			\mathbb{A}^1\setminus\{0\}\ar[r]^-\Delta\ar[d]^-{{}^t\rho_\ell\circ\rho_\ell}&
			(\mathbb{A}^1\setminus\{0\})^2\ar[r]^-{\id\times \rho_\ell}&
			(\mathbb{A}^1\setminus\{0\})^2\ar[d]^-{({}^t\rho_\ell\circ\rho_\ell)\times\id}\\
			\mathbb{A}^1\setminus\{0\}\ar[r]^-\Delta&
			(\mathbb{A}^1\setminus\{0\})^2\ar[r]^-{\id\times\rho_\ell}&
			(\mathbb{A}^1\setminus\{0\})^2.
		}
	$$
	Both compositions are given by the cycle on $(\mathbb{A}^1\setminus\{0\})^3$ defined by $x^\ell=y^\ell=z$.
\end{proof}

\begin{definition}\label{def:motivic_operations}
Let $n\geq 1$.
	We define the \emph{differential}
	$$
		d\colon h_0(W^+_n\otimes\mathbb{G}_m^{+\otimes q})
		\to h_0(W^+_n\otimes\mathbb{G}_m^{+\otimes (q+1)})
	$$
	to be the morphism induced by $d\otimes \id\colon\widehat{\mathbb{W}}^+_{p^{n-1}}\otimes\mathbb{G}_m^{+\otimes q}\to \widehat{\mathbb{W}}^+_{p^{n-1}}\otimes\mathbb{G}_m^{+\otimes (q+1)}$ via Lemma \ref{delta_and_Verschiebung}.
\end{definition}

\begin{lemma}\label{delta_Leibniz}
	The following diagram becomes commutative after applying $h_0$:
	$$
		\xymatrix{
			\widehat{\mathbb{W}}_{p^{n-1}}^{+\otimes 2}\ar[r]^-{\star}\ar[d]^-{d\otimes\id+\id\otimes d}	&
			\widehat{\mathbb{W}}^+_{p^{n-1}}\ar[d]^-{d}\\
			\widehat{\mathbb{W}}_{p^{n-1}}^{+\otimes 2}\otimes\mathbb{G}_m^+\ar[r]^-{\star\otimes\id}&
			\widehat{\mathbb{W}}^+_{p^{n-1}}\otimes\mathbb{G}_m^+.
		}
	$$
\end{lemma}

\begin{proof}
	It suffices to show that the following diagram becomes commutative after composing with the canonical epimorphism $\mathbb{Z}_\tr(\mathbb{A}^1\setminus\{0\})^{\otimes 2}\to h_0 (\widehat{\mathbb{W}}^+_{p^{n-1}}\otimes\mathbb{G}_m^+)$:
	$$
		\xymatrix{
		\mathbb{Z}_\tr(\mathbb{A}^1\setminus\{0\})^{\otimes 2}\ar[r]^-	\mu\ar[d]^-{q_1+q_2}
		&\mathbb{Z}_\tr(\mathbb{A}^1\setminus\{0\})\ar[d]^-\Delta\\
		\mathbb{Z}_\tr(\mathbb{A}^1\setminus\{0\})^{\otimes 3}\ar[r]^-{\mu\otimes \id}
		&\mathbb{Z}_\tr(\mathbb{A}^1\setminus\{0\})^{\otimes 2}.
		}
	$$
	Here, $q_1$ (resp. $q_2$) denotes the morphism $(x,y)\mapsto (x,y,x)$ (resp. $(x,y)\mapsto (x,y,y)$).
	The two compositions are identified with
	\begin{align*}
		&\mathbb{Z}_\tr(\mathbb{A}^1\setminus\{0\})^{\otimes 2}\xrightarrow{\beta}
		\mathbb{Z}_\tr(\mathbb{A}^1\setminus\{0\})^{\otimes 4}\xrightarrow{\mu\otimes\mu}
		\mathbb{Z}_\tr(\mathbb{A}^1\setminus\{0\})^{\otimes 2},
		\\
		&\mathbb{Z}_\tr(\mathbb{A}^1\setminus\{0\})^{\otimes 2}\xrightarrow{\beta}
		\mathbb{Z}_\tr(\mathbb{A}^1\setminus\{0\})^{\otimes 4}\xrightarrow{\mu\otimes (\pr_1+\pr_2)}
		\mathbb{Z}_\tr(\mathbb{A}^1\setminus\{0\})^{\otimes 2}
	\end{align*}
	where $\beta$ is the morphism $(x,y)\mapsto (x,y,x,y)$.
	Therefore it suffices to show that the two morphisms
	$$
		\mu,\pr_1+\pr_2\colon \mathbb{Z}_\tr(\mathbb{A}^1\setminus\{0\})^{\otimes 2}\to \mathbb{Z}_\tr(\mathbb{A}^1\setminus\{0\})
	$$
	coincide after composing with the canonical epimorphism $\mathbb{Z}_\tr(\mathbb{A}^1\setminus\{0\})\twoheadrightarrow h_0\mathbb{G}_m^+$.
	This follows from Lemma \ref{Gm_anti_commutative}.
\end{proof}

\begin{lemma}\label{FdV}
	The following diagram becomes commutative after applying $h_0$:
	$$
		\xymatrix{
			\widehat{\mathbb{W}}^+_{p^{n-1}}\ar[r]^-{V_p}\ar[d]^-{d}	&
			\widehat{\mathbb{W}}^+_{p^n}\ar[d]^-{d}\\
			\widehat{\mathbb{W}}^+_{p^{n-1}}\otimes\mathbb{G}_m^+&
			\widehat{\mathbb{W}}^+_{p^n}\otimes\mathbb{G}_m^+\ar[l]_-{F_p\otimes \id}.
		}
	$$
\end{lemma}

\begin{proof}
	By Lemma \ref{Gm_anti_commutative} (2) and our assumption that $p$ is odd, we have $h_0(\id\otimes V_p)=\id$ on $h_0(\widehat{\mathbb{W}}^+_{p^{n-1}}\otimes\mathbb{G}_m^+)$.
	Therefore it suffices to show that the following diagram is commutative:
	$$
		\xymatrix{
			\widehat{\mathbb{W}}^+_{p^{n-1}}\ar[rr]^-{V_p}\ar[d]^-{d}	&&
			\widehat{\mathbb{W}}^+_{p^n}\ar[d]^-{d}\\
			\widehat{\mathbb{W}}^+_{p^{n-1}}\otimes\mathbb{G}_m^+\ar[r]^-{\id\otimes V_p}&
			\widehat{\mathbb{W}}^+_{p^{n-1}}\otimes\mathbb{G}_m^+&
			\widehat{\mathbb{W}}^+_{p^n}\otimes\mathbb{G}_m^+\ar[l]_-{F_p\otimes \id}.
		}
	$$
	This is further reduced to the commutativity of the following diagram in $\Cor_k$:
	$$
		\xymatrix{
			\mathbb{A}^1\setminus\{0\}\ar[rr]^-{{}^t\rho_p}\ar[d]^-{\Delta}	&&
			\mathbb{A}^1\setminus\{0\}\ar[d]^-{\Delta}\\
			(\mathbb{A}^1\setminus\{0\})^2\ar[r]^-{\id\otimes {}^t\rho_p}&
			(\mathbb{A}^1\setminus\{0\})^2&
			(\mathbb{A}^1\setminus\{0\})^2\ar[l]_-{\rho_p\otimes \id}.
		}
	$$
	Both compositions are given by the cycle on $(\mathbb{A}^1\setminus\{0\})^3$ defined by $x=y=z^p$.
\end{proof}

\begin{theorem}\label{thm:theta}
	Let $X=\Spec A\in \Sm_k^\aff$.
	Then the tuple $(a_\Nis h_0(W^+_\bullet\otimes \mathbb{G}_m^{+\otimes *})(X), F, V, \lambda)$ is a Witt complex over $A$.
	In particular, we have a unique homomorphism of Witt complexes
	$$
		\theta\colon W_\bullet \Omega^*(A)\to a_\Nis h_0(W^+_\bullet\otimes \mathbb{G}_m^{+\otimes *})(X)
	$$
	which extends to a morphism $\theta\colon W_n \Omega^q\to a_\Nis h_0(W^+_n\otimes \mathbb{G}_m^{+\otimes q})$ in $\Sh_\Nis(\Sm_k)$.
\end{theorem}

\begin{proof}
	By Lemma \ref{Gm_anti_commutative} (5), we have $2d^2=0$ and hence $d^2=0$.
	By Lemma \ref{delta_Leibniz} and Lemma \ref{Gm_anti_commutative} (4), we see that $d$ satisfies the Leibniz rule.
	The relation $FdV=d$ holds by Lemma \ref{FdV}.
	Let us show that the relation
	$$
		F(d\lambda[a])=\lambda[a^{p-1}]\star d\lambda[a]
	$$
	holds for any $a\in A$.
	We use the \emph{injectivity theorem for reciprocity sheaves} \cite[Theorem 6 \& 7]{KSY1}; since $h_0(W^+_n\otimes \mathbb{G}_m^{+\otimes q})$ is a reciprocity presheaf, its Nisnevich sheafification has global injectivity.
	This allows us to assume that $a$ is invertible in $A$.
	In this case, both sides of the claimed formula are represented by the morphism $(a^p,a)\colon \Spec A\to (\mathbb{A}^1\setminus\{0\})^2$.
\end{proof}

\subsection{Compatibility with transfers}

In the last subsection, we have constructed a morphism $\theta\colon W_n \Omega^q\to a_\Nis h_0(W^+_n\otimes \mathbb{G}_m^{+\otimes q})$ in $\Sh_\Nis(\Sm_k)$.
In this subsection, we prove that $\theta$ is compatible with transfers.

\begin{lemma}[Projection formula]
	Let $X\in \Sm_k$, $Y\in \Sm_X$ and let $\pi\colon Y\to X$ be the structure morphism.
	Let $a\in W_n\Omega^q(X)$ and $b\in W_n\Omega^q(Y)$.
	For any $\alpha\in \Cor_X(X,Y)$, we have
	$$
		\alpha^*(\pi^*a\star b)=a \star \alpha^*b.
	$$
\end{lemma}

\begin{proof}
	Since the transfer structure of $W_n\Omega^q$ is compatible with the trace maps defined in \cite[Theorem 2.6]{Rulling-thesis}, the claim follows from the projection formula for the trace maps.
\end{proof}

\begin{lemma}\label{theta_transfer_1}
	For any $X=\Spec A\in \Sm_k$ and $\alpha\in \Cor_X(X, \mathbb{A}^1_X)$, the following diagram is commutative:
	$$
	\xymatrix{
		W_n\Omega^q(A[t])\ar[r]^-{\theta}\ar[d]^-{\alpha^*}			&a_\Nis h_0(W_n^+\otimes\mathbb{G}_m^{+\otimes q})(A[t])\ar[d]^-{\alpha^*}\\
		W_n\Omega^q(A)\ar[r]^-{\theta}		&a_\Nis h_0(W_n^+\otimes\mathbb{G}_m^{+\otimes q})(A).
	}
	$$
\end{lemma}

The following proof is inspired by the proof of \cite[Proposition 5.19]{KP11}.

\begin{proof}
	We prove by induction on $q$.
	For $q=0$, this follows from Corollary \ref{cor:p_typical_comparison}.
	Suppose that $q\geq 1$.
	We prove by induction on $n$.
	We say that an element $\omega\in W_n\Omega^q_{A[t]}$ is \emph{traceable} if $\alpha^*\theta(\omega)=\theta(\alpha^*\omega)$ holds.
	Since $F$, $V$, and $d$ are compatible with $\alpha^*$, if $\omega$ is traceable then so is $F(\omega)$, $V(\omega)$, and $d\omega$.
	The group $W_n\Omega^q_{A[t]}$ is a quotient of $\Omega^q_{W_n(A[t])}$ and hence generated by elements of the form
	$$
		\omega = V^{j_0}([a_0])\star dV^{j_1}([a_1])\star \dots\star dV^{j_q}([a_q])
	$$
	where $a_0,\dots,a_q\in A[t]$.
	Let us prove that $\omega$ is traceable.
	
	\noindent
	(1)	If $j_0=\dots=j_q=0$, then we can write $\omega$ as a $\mathbb{Z}$-linear combination of elements of the form
	\begin{align*}
		\eta &= [c_0t^{k_0}]\star d[c_1]\star \dots\star d[c_q],\\
		\xi &= [c_0t^{k_0}]\star d[t]\star d[c_1]\star \dots\star d[c_{q-1}]
	\end{align*}
	where $c_0,\dots,c_q\in A$.
	First we show that $\eta$ is traceable.
	By the projection formula, it suffices to show that $[c_0t^{k_0}]$ is traceable.
	This follows from the case $q=0$.
	Next we show that $\xi$ is traceable.
	We write $k_0+1=mp^e$ with $p\nmid m$.
	Then we have
	\begin{align*}
		F^e(d[t^m])&=F^{e-1}([t^{m(p-1)}]\star d[t^m])\\
		&=F^{e-2}([t^{mp(p-1)}]\star [t^{m(p-1)}]\star d[t^m])\\
		&=\cdots=[t^{m(p^e-1)}]\star d[t^m]\\
		&=m[t^{mp^e-1}]\star d[t] = m[t^{k_0}]\star d[t]
	\end{align*}
	and hence
	$$
		\xi = m^{-1}F^e (d[t^m])\star [c_0]\star d[c_1]\star \dots\star d[c_{q-1}].
	$$
	By the projection formula, it suffices to show that $m^{-1}F^e (d[t^m])$ is traceable.
	This follows from the case $q=0$.
	
	\noindent
	(2)	If $j_0\geq 1$, then we have $\omega = V(\omega')$ for some $\omega'\in W_{n-1}\Omega^q_{A[t]}$, so the induction hypothesis on $n$ shows that $\omega$ is traceable.
	
	\noindent
	(3)	If $j_0=0$ and $j_1\geq 1$, then we have
	$$
		[a_0]\star dV^{j_1}([a_1])=d([a_0]\star V^{j_1}([a_1])) - V^{j_1}([a_1]) \star d([a_0]).
	$$
	We set $\omega_0=[a_0]\star V^{j_1}([a_1])$, $\omega_1=V^{j_1}([a_1]) \star d([a_0])$, and $\omega_2=dV^{j_2}([a_2])\star \dots\star dV^{j_q}([a_q])$, so that $\omega=d(\omega_0\star \omega_2)-\omega_1\star \omega_2$.
	Then $d(\omega_0\star \omega_2)$ is traceable by the induction hypothesis on $q$, and $\omega_1\star\omega_2$ is traceable by (1).
	
	\noindent
	(4)	If $j_0=0$ and $j_i\geq 1$ for some $2\leq i\leq q$, then $\omega$ is traceable by the same argument as in (3).
\end{proof}

\begin{lemma}\label{theta_transfer_2}
	Let $X=\Spec A\in \Sm_k^\aff$ and set $A_N=A[t_1,\dots,t_N]$.
	For any $\alpha\in \Cor_X(X,\mathbb{A}^N_X)$, the following diagram is commutative:
	$$
	\xymatrix{
		W_n\Omega^q(A_N)\ar[r]^-{\theta}\ar[d]^-{\alpha^*}		&a_\Nis h_0(W_n^+\otimes\mathbb{G}_m^{+\otimes q})(A_N)\ar[d]^-{\alpha^*}\\
		W_n\Omega^q(A)\ar[r]^-{\theta}		 &a_\Nis h_0(W_n^+\otimes\mathbb{G}_m^{+\otimes q})(A).
	}
	$$
\end{lemma}

\begin{proof}
	Since the target of $\theta$ has global injectivity, we may assume that $A=K$ where $K$ is a function field.
	We proceed by induction on $N$, the case $N=1$ being Lemma \ref{theta_transfer_1}.
	We may assume that $\alpha = [x]$ where $x$ is a closed point of $\mathbb{A}^N_K$.
	Let $\pi\colon \mathbb{A}^N_K\to \mathbb{A}^{N-1}_K$ be the canonical projection and set $y= \pi(x)$.
	Then $[y]$ defines an element of $\Cor_K(\Spec K, \mathbb{A}^{N-1}_K)$.
	Since $x$ is a closed point of $\mathbb{A}^1_{k(y)}$, it is defined by some monic equation
	$$
		t_N^m+a_{m-1}t_N^{m-1}+\dots+a_0=0
	$$
	where $a_0,\dots,a_{m-1}\in k(y)$.
	Let $b_i$ be a lift of $a_i$ to $K_{N-1}$ for $i=0,\dots,m-1$.
	Define $\Gamma\subset \mathbb{A}^N_K$ by the equation
	$$
		t_N^m+b_{m-1}t_N^{m-1}+\dots+b_0=0.
	$$
	Then $[\Gamma]$ defines an element of $\Cor_{\mathbb{A}^{N-1}_K}(\mathbb{A}^{N-1}_K,\mathbb{A}^N_K)$ with $[x]=[\Gamma]\circ[y]$.
	Therefore the claim follows from the induction hypothesis and Lemma \ref{theta_transfer_1}.
\end{proof}

\begin{lemma}\label{theta_transfer_3}
	For any $X,Y\in \Sm_k$ and $\alpha\in \Cor_k(X,Y)$, the following diagram is commutative:
	$$
	\xymatrix{
		W_n\Omega^q(Y)\ar[r]^-{\theta}\ar[d]^-{\alpha^*}		&a_\Nis h_0(W_n^+\otimes\mathbb{G}_m^{+\otimes q})(Y)\ar[d]^-{\alpha^*}\\
		W_n\Omega^q(X)\ar[r]^-{\theta}		&a_\Nis h_0(W_n^+\otimes\mathbb{G}_m^{+\otimes q})(X).
	}
	$$
	In other words, $\theta$ is a morphism in $\Sh_\Nis(\Cor_k)$.
\end{lemma}

\begin{proof}
	The finite correspondence $\alpha$ can be factored as
	$$
		X\xrightarrow{(\id,\alpha)}X\times Y\xrightarrow{\pr_2} Y.
	$$
	Therefore we may assume that $Y\in \Sm_X$ and $\alpha\in \Cor_X(X,Y)$.
	Since $\theta$ is a morphism of Nisnevich sheaves, we may assume that $X=\Spec A$, $Y=\Spec B$.
	Take a closed immersion $\Spec B\to \Spec A_N$ where $A_N=A[t_1,\dots,t_N]$.
	Since $W_n\Omega(A_N)\to W_n\Omega(B)$ is surjective, we may assume that $B=A_N$, so the claim follows from Lemma \ref{theta_transfer_2}.
\end{proof}

\subsection{Comparison}

\begin{definition}
	Consider the sections
	$$
		[t]\colon \mathbb{Z}_\tr(\mathbb{A}^1\setminus\{0\})\to W_n,\quad t\colon \mathbb{Z}_\tr(\mathbb{A}^1\setminus\{0\})\to \mathbb{G}_m
	$$
	where $t$ is the coordinate function of $\mathbb{A}^1\setminus\{0\}$.
	One can easily see that these morphisms descend to
	$$
		[t]\colon \omega_!W^+_n\to W_n, \quad t\colon \omega_!\mathbb{G}^+_m \to \mathbb{G}_m.
	$$
	Taking a tensor product and composing with
	$$
		W_n\otimes\mathbb{G}_m^{\otimes q}\to W_n\Omega^q;\quad x\otimes a_1\otimes\dots\otimes a_q\mapsto x\star\dlog[a_1]\star\dots\star\dlog[a_q],
	$$
	we get a morphism
	$$
		\eta\colon \omega_!(W_n^+\otimes \mathbb{G}_m^{+\otimes q})\to W_n\Omega^q
	$$
	in $\PSh(\Cor_k)$.
	In other words, $\eta$ is the morphism induced by the section
	$[t_0]\star \dlog[t_1]\star \cdots\star \dlog[t_q]$ of $W_n\Omega^q$ on $(\mathbb{A}^1\setminus\{0\})^{q+1}$, where $t_0,\dots,t_q$ are the coordinate functions of $(\mathbb{A}^1\setminus\{0\})^{q+1}$.
\end{definition}

\begin{lemma}
	The morphism $\eta\colon \omega_!(W_n^+\otimes \mathbb{G}_m^{+\otimes q})\to W_n\Omega^q$ descends to $h_0 (W_n^+\otimes \mathbb{G}_m^{+\otimes q})$.
\end{lemma}

Note that $\eta$ further factors through $a_\Nis h_0 (W_n^+\otimes \mathbb{G}_m^{+\otimes q})$ since $W_n\Omega^q$ is a Nisnevich sheaf.

\begin{proof}
	We identify $\eta$ with the corresponding section of $W_n\Omega^q$ on $(\mathbb{A}^1\setminus\{0\})^{q+1}$.
	Fix $\varepsilon\in \mathbb{Q}_{>0}$.
	Let $\mathcal{W}=(\mathbb{P}^1,\varepsilon[0]+(p^{n-1}+\varepsilon)[\infty])$ and $\mathcal{G}=(\mathbb{P}^1,\varepsilon[0]+\varepsilon[\infty])$.
	It suffices to show that for any connected $X\in \Sm_k$ and
	$$
		\alpha\in \Cor_k(X\otimes (\mathbb{P}^1,\{1\}), \mathcal{W}\otimes \mathcal{G}^{\otimes q}),
	$$
	we have $(i_0^*-i_\infty^*)\alpha^*\eta=0 \in W_n\Omega^q(X)$.
	Since the homomorphism $W_n\Omega^q(X)\to W_n\Omega^q(k(X))$ is injective, we may assume that $X=\Spec K$ where $K$ is a function field.
	Moreover, we may assume that $\alpha=[\Gamma]$ for some integral closed subscheme $\Gamma\subset \square_K\times_K(\mathbb{A}^1_K\setminus\{0\})^{q+1}$ which is finite surjective over $\square_K$.
	Let $C$ denote the normalization of $\Gamma$.
	Let $g\colon C\to \square_K$ and $(f_0,\dots,f_q)\colon C\to (\mathbb{A}^1_K\setminus\{0\})^{q+1}$ denote the canonical projections.
	Then $\alpha\circ(i_0-i_\infty)$ is equal to the composite
	$$
		\Spec K\xrightarrow{\div(g)} C\xrightarrow{(f_0,\dots,f_q)} (\mathbb{A}^1_K\setminus\{0\})^{q+1}.
	$$
	Therefore it suffices to show that
	$$
		\textstyle\sum_{x\in C^{(1)}} n_x\Tr_{K(x)/K}([f_0]\star\dlog[f_1]\star\dots\star \dlog[f_q])=0
	$$
	where $\div(g)=\sum_{x\in C^{(1)}} n_x\{x\}$.
	Using the residue map (\cite{Rulling-thesis}, \cite{Rulling-erratum}), we can rewrite the left hand side as
	$$
		\textstyle\sum_{x\in C^{(1)}} \Res_{C/K,x}([f_0]\star\dlog[f_1]\star\dots\star \dlog[f_q]\star \dlog[g]).
	$$
	By the residue theorem (\cite[Theorem 2.19]{Rulling-thesis}, \cite[Theorem 2]{Rulling-erratum}), the above sum is zero if we replace $C$ with the regular projective model $\overline{C}$ of $C$.
	Therefore it suffices to show that $[f_0]\star\dlog[f_1]\star\dots\star \dlog[f_q]\star \dlog[g]$ is regular at every $x\in \overline{C}\setminus C$.
	Let $\pi$ be a local parameter at $x$ and write $f_i=u_i\pi^{m_i}$ and $g=1+u\pi^m$ where $u_i,u\in \mathcal{O}_{\overline{C},x}^\times$.
	Note that we always have $m>0$.
	The admissibility of $\Gamma$ implies
	$$
		m\geq \varepsilon \max\{m_0,0\}+(p^{n-1}+\varepsilon)\max\{-m_0,0\}+\varepsilon (|m_1|+\cdots+|m_q|)>-p^{n-1}m_0.
	$$
	Hence the result follows from the next lemma.
\end{proof}

\begin{lemma}
	Let $(K,v)$ be a henselian discrete valuation field of characteristic $p$, $\mathcal{O}_K$ be its valuation ring and $\mathfrak{m}_K$ be the maximal ideal of $\mathcal{O}_K$.
	Let $f_0,\dots,f_q\in K^\times$, $a\in \mathfrak{m}_K$ and assume that $v(a)>-p^{n-1}v(f_0)$.
	Then we have $[f_0]\star\dlog[f_1]\star\cdots\star\dlog[f_q]\star\dlog[1+a]\in W_n\Omega^q(\mathcal{O}_K)$.
\end{lemma}

\begin{proof}
	Let $\pi$ be a uniformizer of $K$.
	Write $f_i=u_i\pi^{m_i}$ and $a=u\pi^m$ where $u_i, u\in \mathcal{O}_K^\times$.
	Then we have $m>-p^{n-1}m_0$.
	By \cite[Lemma 3.4]{Rulling-thesis}, we can write
	\begin{align*}
		[1+a]&=[1]+\textstyle\sum_{i=0}^{n-1}V^i([\pi^m]\star c_i)
	\end{align*}
	where $c_i\in W_n(\mathcal{O}_K)$.
	Therefore it suffices to show that
	$$
		[f_0]\star\dlog[f_1]\star\cdots\star\dlog[f_q]\star dV^i([\pi^m]\star c_i)
	$$
	is regular for $i=0,1,\dots,n-1$.
	By the Leibniz rule, it suffices to show that
	\begin{enumerate}
		\item	$[f_0]\star\dlog[f_1]\star\cdots\star\dlog[f_q]\star V^i([\pi^m]\star c_i)$,
		\item	$d[f_0]\star\dlog[f_1]\star\cdots\star\dlog[f_q]\star V^i([\pi^m]\star c_i)$
	\end{enumerate}
	are regular.
	For (1), we have
	\begin{align*}
		&[f_0]\star\dlog[f_1]\star\cdots\star\dlog[f_q]\star V^i([\pi^m]\star c_i)\\
		=&V^i(F^i[f_0]\star\dlog[f_1]\star\cdots\star\dlog[f_q]\star [\pi^m]\star c_i)\\
		=&V^i([u_0^{p^i}\pi^{m+p^im_0}]\star\dlog[f_1]\star\cdots\star\dlog[f_q]\star c_i)
	\end{align*}
	and hence the claim follows from $m>-p^im_0$.
	For (2), we have
	$d[f_0]=[f_0]\star \dlog [f_0]$
	and hence the claim follows from the regularity of (1).
\end{proof}

\begin{lemma}\label{eta_compatibility}
	The morphism $\eta\colon h_0 (W^+_n\otimes \mathbb{G}_m^{+\otimes q})\to W_n\Omega^q$ is compatible with $\star$, $d$, $F$, $V$, and $\lambda$.
\end{lemma}

\begin{proof}
	The claim for $\star,F,V,\lambda$ follows from Proposition \ref{prop:ring-comparison}.
	Let us prove the claim for $d$.
	Define $\delta\colon (\mathbb{A}^1\setminus\{0\})^{q+1}\to (\mathbb{A}^1\setminus\{0\})^{q+2}$ by
	$(t_0,\dots,t_q)\mapsto (t_0,t_0,t_1,\dots,t_q)$.
	Then we have the following commutative diagram:
	$$
		\xymatrix{
			\mathbb{Z}_\tr(\mathbb{A}^1\setminus\{0\})^{\otimes (q+1)}\ar[r]^-\delta\ar@{->>}[d]	&\mathbb{Z}_\tr(\mathbb{A}^1\setminus\{0\})^{\otimes (q+2)}\ar@{->>}[d]\\
			h_0 (W_n^+\otimes \mathbb{G}_m^{+\otimes q})\ar[r]	^-d							&h_0 (W_n^+\otimes \mathbb{G}_m^{+\otimes (q+1)}).
		}
	$$
	Therefore it suffices to show that $\delta^*\eta = d\eta$ holds as a section of $W_n\Omega^{q+1}$ on $(\mathbb{A}^1\setminus\{0\})^{q+1}$.
	This follows from the equality $[t]\star \dlog[t]=d[t]$ in $W_n\Omega^1(\mathbb{A}^1\setminus\{0\})$.
\end{proof}

\begin{lemma}\label{lem:theta_of_eta}
	We have $\theta(\eta) = \id$, where the right hand side denotes the section of $a_\Nis h_0(W_n^+\otimes\mathbb{G}_m^{+\otimes q})$ on $(\mathbb{A}^1\setminus\{0\})^{q+1}$ represented by the identity morphism.
\end{lemma}

\begin{proof}
	The section $\eta$ can be written as
	$$
		\eta = [t_0t_1^{-1}\dots t_q^{-1}]\star d[t_1]\star\dots\star d[t_q]
	$$
	where $t_0,\dots,t_q$ are the coordinate functions of $(\mathbb{A}^1\setminus\{0\})^{q+1}$.
	Since $\theta$ is a morphism of Witt complexes, $\theta(\eta)$ is given by the same formula in $h_0(W_n^+\otimes \mathbb{G}_m^{+\otimes q})((\mathbb{A}^1\setminus\{0\})^{q+1})$.
	The element $[t_0t_1^{-1}\dots t_q^{-1}]$ is represented by the morphism
	$$
		t_0t_1^{-1}\dots t_q^{-1}\colon (\mathbb{A}^1\setminus\{0\})^{q+1}\to \mathbb{A}^1\setminus\{0\}
	$$
	and the element $d[t_i]$ is represented by the morphism
	$$
		(\mathbb{A}^1\setminus\{0\})^{q+1}\xrightarrow{t_i}\mathbb{A}^1\setminus\{0\}\xrightarrow{\Delta} (\mathbb{A}^1\setminus\{0\})^2.
	$$
	By definition of the multiplication on $h_0(W_n^+\otimes \mathbb{G}_m^{+\otimes *})$, we see that $\theta(\eta)$ is represented by the identity morphism.
\end{proof}

\begin{theorem}\label{thm:rep-WOmega}
Let $k$ be a perfect field of characteristic $p\geq 3$.
Then there is an isomorphism
$$
	\eta\colon a_{\Nis} h_0 (W_n^+ \otimes \mathbb{G}_m^{+\otimes q})\xrightarrow{\sim}W_n\Omega^q
$$
in $\Sh_\Nis(\Cor_k)$.
\end{theorem}

\begin{proof}
	Lemma \ref{lem:theta_of_eta} shows that $\theta\circ \eta=\id$.
	Let $X=\Spec A\in \Sm_k^\aff$.
	Since $W_n\Omega^q(A)$ is initial in the category of Witt complexes over $A$, and $\eta\circ \theta$ gives an endomorphism of $W_n\Omega^q(A)$ as a Witt complex over $A$ by Lemma \ref{eta_compatibility}, we get $\eta\circ \theta=\id$.
\end{proof}

\printbibliography

\end{document}